\documentclass[a4paper,11pt]{amsart}
\usepackage{times}
\usepackage{amsmath,amsopn,amscd,amsthm}
\usepackage{graphicx,xspace}
\usepackage{epsfig}
\usepackage{enumitem}
\usepackage{xcolor}
\usepackage{tikz}
\usetikzlibrary{decorations.pathreplacing}

\usepackage{srcltx}
\usepackage[T1]{fontenc}
\usepackage[utf8]{inputenc}
\usepackage[notcite, notref,color,final]{showkeys}
\definecolor{refkey}{rgb}{0,0,1}
\definecolor{labelkey}{rgb}{0,0,1}

\usepackage{hyperref}
\usepackage[capitalise]{cleveref}
 \def\e{\, \,}

\newtheorem{lemma}{Lemma}[section]
\newtheorem{theorem}[lemma]{Theorem}
\newtheorem{corollary}[lemma]{Corollary}
\newtheorem{proposition}[lemma]{Proposition}

\newtheorem{conjecture}[lemma]{Conjecture}

\usepackage[margin=1in]{geometry}
\usepackage{todonotes}
\theoremstyle{remark}
\newtheorem{remark}[lemma]{Remark}
\newtheorem{example}[lemma]{Example}
\newtheorem{definition}[lemma]{Definition}

\newcommand{\QQ}{\mathbb{C}}

\newcommand{\Sym}[1]{\mathfrak{S}_{#1}}

\newcommand{\hata}{\bar{\alpha}}
\newcommand{\hatb}{\bar{\beta}}
\newcommand{\hats}{\ensuremath{\bar{\sigma}}\xspace}
\newcommand{\downi}[1]{\ensuremath{ {\downarrow (#1)}}\xspace}
\newcommand{\downalphai}[1]{\ensuremath{\downarrow (\eye_{#1})}\xspace}

\newcommand{\eps}{\varepsilon}

\DeclareMathOperator{\Odd}{Odd}
\DeclareMathOperator{\Even}{Even}

\author{Valentin F\'eray}
\address{Institut f\"ur Mathematik, Universit\" at Z\" urich, Z\" urich, Switzerland}
\email{valentin.feray@math.uzh.ch}
\author{Amarpreet Rattan}
\address{Department of Economics, Mathematics and Statistics Birkbeck College,
University of London Malet Street, London WC1E 7HX, UK}
\email{arattan@gmail.com}
\title[On products of cycles]
{On products of long cycles:\\
short cycle dependence and separation probabilities}

\begin{document}

\begin{abstract}
We present various results on multiplying cycles in the symmetric group.  One
result is a generalisation of the following theorem of Boccara (1980):  the number of
ways of writing an odd permutation in the symmetric group on $n$ symbols as a product
of an $n$-cycle and an $(n-1)$-cycle is independent of the permutation chosen.  
We give a number of different approaches of our generalisation.
One partial proof uses an inductive method
which we also apply to other problems.  In particular, we give a formula for the distribution of
the number of cycles over all products of cycles of fixed lengths.  
Another application is related to the recent notion of separation probabilities for permutations
introduced by Bernardi, Du, Morales and Stanley (2014).
\end{abstract}  

\maketitle
\section{Introduction}\label{sec:intro}
We begin by introducing standard definitions and notation.

Let $n$ be a positive integer.  A \emph{partition $\lambda$ of $n$} is a weakly decreasing sequence of positive
integers $(\lambda_1, \cdots, \lambda_r)$ whose sum is $n$, and we write $\lambda \vdash n$ or $|\lambda| = n$.  
We call the $\lambda_i$ the \emph{parts} of $\lambda$.  The number of parts
$r$ is the \emph{length} of $\lambda$ and we write this as 
$\ell(\lambda)$.
We will also use the exponential notation $\lambda=(1^{m_1},2^{m_2},\dots)$,
where $m_i$ (or $m_i(\lambda)$) is the number of parts of $\lambda$ equal to $i$.
For any non-negative integer $a < n$, a partition of the form $(n-a,
1^a)$ is called a \emph{hook}.

Let $\Sym{n}$ be the symmetric group on the symbols 
$[n] : = \{1,2, \dots, n\}$. 
It is well known that the conjugacy classes of
$\Sym{n}$ are indexed by partitions of $n$;  a permutation's weakly
decreasing list of cycle lengths determines the partition giving its conjugacy
class.  We call this partition the permutation's \emph{cycle type}, and we denote by
$C_\lambda$ the class of permutations with type $\lambda$.
Any permutation whose cycle type  is a hook
is called a \emph{cycle}, and, more specifically, a permutation with cycle type
$(n-i, 1^i)$ is an $(n-i)$-cycle.  A \emph{full} cycle has type $(n)$. 

Let $\eps$
be the signature of a permutation;  that is,  
$\eps(\sigma) = 1$ if $\sigma$ is an even permutation and $-1$ otherwise.  We
abuse notation and write $\eps(\lambda)$ for a partition $\lambda$, which is the
signature of any permutation of type $\lambda$, and we call the partition
\emph{even} or \emph{odd} accordingly.  It is well known that a partition $\lambda$ is even if and
only if $|\lambda| - \ell(\lambda)$ is an even integer.

For a partition of $\lambda$ of $n$, let $K_\lambda$ be the formal sum of all
elements in the group algebra of $\Sym{n}$ with cycle type $\lambda$.  The set $\{K_\lambda : \lambda
\vdash n\}$ forms a basis for the centre of the group algebra.  Thus, for partitions $\mu$ and $\nu$ of $n$, we have
an expansion
\begin{equation*}
	K_\mu \cdot K_\nu = \sum_{\lambda \vdash n} c^\lambda_{\mu,\nu}
	K_\lambda.
\end{equation*}
The numbers $c^\lambda_{\mu,\nu}$ are known as the \emph{connection
coefficients} of $\Sym{n}$ and are the central objects of this article.  Equivalently, the connection coefficient
$c^\lambda_{\mu,\nu}$ gives the number of ways of writing a fixed permutation of cycle
type $\lambda$ as the product of two permutations of cycle type $\mu$ and $\nu$,
respectively.  It is this latter combinatorial interpretation of the connection
coefficients that we focus on in this article.

Finding $c^\lambda_{\mu, \nu}$ is an important problem in algebraic combinatorics.  See for example the early work of H. Farahat and G. Higman \cite{FarahatHigman1959}, D. Walkup \cite{Walkup1979}, R. Stanley \cite{StanleyFactorizationIntoNCycles}, D. Jackson \cite{JacksonPermutationsCharacters} and I. Goulden
and D. Jackson \cite{GouldenJackson1992}.  The connection coefficients of the
symmetric group also appear in contexts outside of algebra or combinatorics.
In particular, P. Hanlon, J. Stembridge and R. Stanley relate them to the
spectra of random matrices in \cite{hanlonstembridgestanley}, while D. Jackson
and T. Visentin \cite{JacksonVisentiCharacter} 
establish a connection with rooted maps embeddable on
orientable surfaces.

A rather general formula for $c^\lambda_{\mu, \nu}$ is due to A. Goupil
and G. Schaeffer \cite{GoupilSchaefferStructureCoef}. 
But this result, like most results before it, requires that at
least one of the three partitions $\lambda, \mu$ or $\nu$ is $(n)$.
Since
\begin{equation*}
	|C_\lambda| c^\lambda_{\mu, \nu} = |C_\mu| c^\mu_{\lambda, \nu},
\end{equation*}
we see it is immaterial which one of $\lambda, \mu$ or $\nu$ is $(n)$.

The reason why many of the known results require one of $\lambda, \mu$ or
$\nu$ to be $(n)$ is that a common tool to compute these numbers, which is central in a number of the articles above, 
 is the well-known expression relating the connection coefficients to the irreducible
 character values of the symmetric group, contained in Proposition
 \ref{PropForm3Char} below.  This expression,  which is not specific to the symmetric group,
 is in general not tractable for arbitrary partitions.  
When one of the partitions is $(n)$, however, 
there is a substantial simplification in the formula given in Proposition~\ref{PropForm3Char}:
only characters indexed by hooks contribute to the result.
This well-known simplification is contained in Lemma \ref{LemCharNCycle} below.\bigskip

The starting point of this paper is the following observation:
when one of the partitions corresponds to an $n$-cycle and another corresponds to an $(n-a)$-cycle for small $a$ (say $\mu$ and $\nu$, respectively), 
the formula expressing $c^\lambda_{\mu, \nu}$ in terms of characters
becomes even simpler.

For partitions $\lambda \vdash n$ and $\rho \vdash a$ with $n\ge a$,
we denote by $\nu_\rho(\lambda): =
c^{\lambda}_{(n-a, \rho), (n)}$;  i.e. $\nu_\rho(\lambda)$  is 
the number of ways to write a given permutation $\sigma$ of type $\lambda$
as the product of a $n$-cycle $\alpha$ and a permutation $\beta$
of type $(n-a) \cup \rho$.
We present the following theorem.

\begin{theorem}
    Let $n$ and $a < n$ be positive integers.  Fix a partition $\rho$ of $a$.
    There exists an explicit polynomial $Z_\rho$ in the variables $n$ and
    $m_1,\ldots,m_{a-1}$ such that, for every partition $\lambda$ of $n$
    with $m_i$ cycles of length $i$ (for $1\leq i\leq a-1$), one has that
    \begin{equation}
        \nu_\rho(\lambda) = \big(1+ (-1)^a \eps(\rho) \eps(\lambda) \big)
        \, (n-a-1)! \, Z_\rho(n,m_1,\ldots,m_{a-1}).
        \label{EqDependence}
    \end{equation}
    \label{ThmDependence}
\end{theorem}
That $\nu_\rho(\lambda)$ vanishes if
$(-1)^a \eps(\rho) \eps(\lambda)=-1$ is the consequence of a simple parity argument
on permutation signs.
The surprising aspect of Theorem \ref{ThmDependence} is 
that $\nu_\rho(\lambda)$ depends only on $|\lambda|$, $\eps(\lambda)$, $m_1(\lambda)$, \dots, $m_{a-1}(\lambda)$ and not on the larger multiplicities of $\lambda$.

Theorem \ref{ThmDependence} can be seen as a generalisation of Theorem
\ref{thm:boc} by Boccara \cite{boccara}.
\begin{theorem}[Boccara]
	\label{thm:boc}
	Let $n$ be a positive integer and $\lambda$ an odd partition of $n$.  Then
	\begin{equation*}
		c^{\lambda}_{(n), (n-1, 1)} = 2 (n-2)!.
	\end{equation*}
\end{theorem}
A combinatorial proof of Theorem \ref{thm:boc} is given in Cori, Marcus and Schaeffer \cite{CMS}.

Theorem \ref{ThmDependence} can be established
using characters of the symmetric groups, which we do in Section \ref{SectCharacters}.

It is natural to look for a combinatorial explanation of it.
We provide two approaches to this question:
\begin{itemize}
    \item in Section \ref{SectBijection}, we give a purely bijective
        proof of Theorem \ref{ThmDependence} in the case $\rho=(1^a)$, extending the work of Cori \emph{et al.} \cite{CMS};
    \item in Section \ref{SectInduction}, we present an inductive proof
        (assuming Boccara's theorem),
        also for the case $\rho=(1^a)$, using only combinatorial arguments.
\end{itemize}
It seems that the case $\rho=(1^a)$ has a particular structure and
we shall refer to it as the {\em hook-case}.
Finding a combinatorial proof for a general partition $\rho$ remains an open problem.

Interestingly, each method allows to compute explicitly $\nu_{1^a}(\lambda)$
for small values of $a$ and leads to different expressions for $\nu_{1^a}(\lambda)$.
For instance, our inductive approach leads to a compact expression of $\nu_{1^a}(\lambda)$
in terms of symmetric functions; see Theorem \ref{ThmSym}.
In addition, the case where $\lambda$ has no parts of length smaller than $a$,
except fixed points, leads to a particularly compact and elegant formula
(proved in Section \ref{SubsectSmallValuesBijection}).
\begin{proposition}\label{PropA2New}
    Let $a \geq 2$ be an integer and $\sigma$ a permutation in $\Sym{n}$.
    Assume $\sigma$ has no cycles of length smaller than $a$,
    except fixed points (whose number is denoted by $b$).
    The number of ways to write $\sigma$ as the product of an $n$-cycle
    and an $(n-a)$-cycle is
    \[\big(1+(-1)^{a}\eps(\sigma)\big) \frac{(n-a-1)!(n-b)!}{a!(n-a-b+1)!}.\]
\end{proposition}

An advantage of combinatorial proofs is that they can often be used to refine enumerative results,
taking into account statistics that cannot be studied easily with the character approach.
We illustrate this fact by studying {\em separation probabilities}
in products of $n$-cycles with $(n-a)$-cycles 
(here, we can also deal with the case $a=0$).

The notion of separable permutations has been introduced in a recent paper
of Bernardi, Du, Morales and Stanley \cite{BernardiEtAlSeparation}.
The main question is to compute the probability that 
given elements are in the same or different cycles in
a product of uniform random permutations of given types
(see Section \ref{SectSeparability} for a formal definition).
For example, in Proposition \ref{Prop1KSep},
we give an explicit formula for the
following problem, formulated by Stanley in
\cite[pages 54-71]{StanleyTalkPP2010}:
what is the probability that $1,2,\dots,k$ lie in different cycles
of $\alpha \cdot \beta$, where $\alpha$ and $\beta$
are a random full cycle and an $(n-a)$-cycle, respectively, in $\Sym{n}$?
Our method is to establish an induction relation for separation probabilities
(Theorem \ref{ThmRecBonaPb}).
This formula is based on the inductive proof of the hook-case Theorem \ref{ThmDependence}.
\bigskip

Another advantage of the inductive method is that, in principle, 
it can be generalised to factorisation problems where we do not require any partition
to be $(n)$, although computations become more cumbersome.
We give two examples of this kind of results in Section \ref{sec:cycles}:
\begin{itemize}
    \item an explicit formula for the distribution of the number of cycles in the product of two cycles of given lengths (Theorem \ref{thm:maincount});
    \item an involved explicit formula for the separation probability for the
	    product of two $(n-1)$-cycles (Theorem \ref{thm:bigthm}). 
        This leads us to an appealing conjecture (Conjecture 
        \ref{ConjGenSymmetry}), extending a symmetry
        property proved by the four authors mentioned above, given in \cite[Eq. (1)]{BernardiEtAlSeparation}.
\end{itemize}



Our final comment is that our methods are
elementary and combinatorial.  With the exception of Section
\ref{SectCharacters} where we give our character explanation of Theorem
\ref{ThmDependence}, and Section \ref{sec:symfunc} where we give a result in the
ring of symmetric functions, we do not use algebraic tools or combinatorial
objects significantly different from factorisations in the symmetric group.\bigskip\bigskip

{\em Outline of the paper.}
Sections \ref{SectCharacters}, \ref{SectBijection} and \ref{SectInduction}
are devoted to the three approaches of Theorem \ref{ThmDependence}:
a representation-theoretic proof for general $\rho$,
a bijective proof for the hook-case,
and an inductive proof also for the hook-case, respectively.  In addition, at the
end of Section \ref{SectInduction}, we give a formula for a symmetric function
related to multiplying cycles in the symmetric groups.

In Section \ref{SectSeparability}, we compute the separation probability
of the product of a full cycle and a $(n-a)$-cycle.  

In Section \ref{sec:realinduction}, we present a number of lemmas 
needed for the inductive method.  These are then used to prove the theorems in
Section \ref{sec:cycles} mentioned above.


\section{Character explanation} \label{SectCharacters}
\subsection{Character values and connection coefficients}
Let us consider the family of symmetric groups $\Sym{n}$ (for each $n \ge 1$).
It is well known (see, {\em e.g.}, \cite[Chapter 2]{Sagan}) that
both conjugacy classes and irreducible representations of $\Sym{n}$ can
be indexed canonically by partitions of $n$,
so the character table of $\Sym{n}$ is a collection of numbers
$\chi^\lambda(\mu)$, where $\lambda$ and $\mu$ are taken over all partitions of $n$
and are, respectively, the indices of the irreducible representation
and the conjugacy class.

While the following formula is difficult to attribute to a particular author,
character values of symmetric groups are a classical tool to compute
connection coefficients of the symmetric group (see \cite[Lemma
3.3]{JacksonVisentiCharacter}).
\begin{proposition}
    For any triple of partitions $(\lambda, \mu, \nu)$ of the same integer positive $n$, one has that
    \begin{equation}
    c_{\mu,\nu}^\lambda = \frac{n!}{z_\mu z_\nu} \sum_{\pi \vdash n} 
    \frac{\chi^\pi(\lambda) \chi^\pi(\mu) \chi^\pi(\nu)}{\chi^\pi(1^n)},
\end{equation}
where $z_\mu$ is the integer $\prod_{i \ge 1} i^{m_i(\mu)} m_i(\mu)!$.
    \label{PropForm3Char}
\end{proposition}
This proposition has been widely used in the last thirty years to obtain explicit 
expressions for connection coefficients,
especially in the case where one 
of the partitions $\lambda$, $\mu$ or $\nu$ is $(n)$;
see {\em e.g.} \cite{GoupilSchaefferStructureCoef} and references therein.

\subsection{Lemmas on character values}

 Two common methods of computing the characters of the symmetric group are found in
 the literature. One method is to use a change of basis in the symmetric function ring 
\cite{FrobeniusRepresentationSn}, and the second is the 
Murnaghan-Nakayama combinatorial rule,  which is in fact due
to Littlewood and Richardson \cite[§ 11]{LittlewoodRichardsonCharactersSn} (for
a modern treatment see \cite{Sagan} or \cite[§ 1.7]{Macdonald}).

The following lemma is already well known.
\begin{lemma}
    Let $\pi$ be a partition of $n$, one has:
       \[ \chi^\pi((n))=0,\]
    unless $\pi$ is a hook, that is $(n-i,1^i)$ for some $i$ between $0$ and $n-1$.
    \label{LemCharNCycle}
\end{lemma}
This is an immediate consequence of Murnaghan-Nakayama rule and its use can found in almost all
of the aforementioned articles that use Proposition ~\ref{PropForm3Char} to determine
$c^{\lambda}_{\mu, \nu}$.
A consequence of Lemma \ref{LemCharNCycle} is that by setting $\mu=(n)$ in
Proposition \ref{PropForm3Char}, the sum over all partitions of $n$ reduces to a
sum over hooks, and thus a
simple sum of $i$ from $0$ to $n-1$, which is usually much easier to handle.

When $\nu=(n-a) \cup \rho$, where $\rho$ is a partition
of a fixed integer $a$, the sum also simplifies due to the following lemma.
\begin{lemma}
    Let $i$ and $n$ be two integers with $0 \le i <n$, and $\pi$ be the hook partition $(n-i,1^i)$.
    Also, let $a<n$ be a positive integer and $\rho$ a partition of $a$.
    Assume further $a-1 < i <n-a$. Then,
    \[ \chi^\pi \big( (n-a) \cup \rho \big) = 0.\]
    \label{LemCharHookNmaCycle}
\end{lemma}
\begin{proof}
    The Murnaghan-Nakayama rule implies that 
    $ \chi^\pi \big( (n-a) \cup \rho \big) = 0$,
    unless we can find a ribbon $\xi$ of size $n-a$ in the diagram of $\pi$,
    given by
    
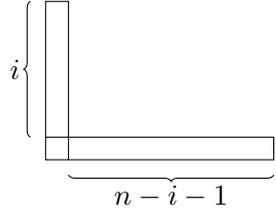
\[\begin{tikzpicture}[auto]
        \draw (0,.3)--(3,.3)--(3,0)--(0,0)--(0,2.1)--(.3,2.1)--(.3,0);
        \draw[decorate, decoration={brace}] (3,-.2)--(.3,-.2) node[below,midway] {$n-i-1$};
        \draw[decorate, decoration={brace}] (-.2,.3)--(-.2,2.1) node[left,midway] {$i$};
    \end{tikzpicture},\]
    so that $\pi$ with $\xi$ removed remains a Young diagram.
    As $a>0$, the ribbon is not the whole diagram.
    Hence it is either in the first row (which is possible only if $n-a \le n-i-1$)
    or in the first column (which is possible only if $n-a \le i$).
    If neither of these conditions is satisfied, there is no ribbon of size $(n-a)$
    in the diagram of $\pi$ and the corresponding character value vanishes,
    as asserted.
\end{proof}

The last lemma we need is due to Littlewood \cite[p. 139]{LittlewoodBookRepresentation}.  See also \cite[Lemma 2.2]{StanleyFactorizationIntoNCycles}:
\begin{lemma}
    Let $i$ and $n$ be two integers with $0 \le i <n$,
    $\lambda$ a partition of $n$, and $\pi$ the hook partition $(n-i,1^i)$.
    Then 
    \[\chi^\pi(\lambda)=\sum_{\rho \vdash i} (-1)^{r_2 +r_4 +\dots}
    \binom{m_1-1}{r_1} \binom{m_2}{r_2} \dots \binom{m_i}{r_i},\]
    \label{LemCharHookExplicit}
    where $m_j$ (respectively $r_j$) is the multiplicity of $j$ in $\lambda$
    (respectively $\rho$) for $1 \leq j \leq i$.
\end{lemma}
\begin{corollary}
	With the notation of Lemma \ref{LemCharHookExplicit} above, if we fix the sign $\eps(\lambda)$,
    then $\chi^\pi(\lambda)$ is a polynomial  in
    the multiplicities $m_j$ for $j \leq i, n-i-1$.
\end{corollary}
\begin{proof}
    The case $i \leq n-i-1$ comes from the explicit expression above.
    The case $i >n-i-1$ follows from the symmetry formula, found in 
    \cite[§1.7, Example 2]{Macdonald}, for characters
    \[\chi^\pi(\lambda)= \epsilon(\lambda) \chi^{\pi'}(\lambda),\]
    where $\pi'=(i+1,1^{n-i-1})$ is the conjugate partition of $\pi$.
\end{proof}

We now have the necessary tools to give a proof of Theorem \ref{ThmDependence}.

\subsection{Proof of Theorem \ref{ThmDependence}}

Fix a partition $\rho$ of a positive integer $a$.
By Proposition~\ref{PropForm3Char}, for any partition $\lambda$ of size $n$
bigger than $2a$,
one has
\[
    \nu_\rho(\lambda)=c^\lambda_{(n),(n-a)\cup \rho}= 
    \frac{n!}{n\cdot (n-a) \cdot z_\rho} \sum_{\pi \vdash n}
    \frac{\chi^\pi(\lambda) \chi^\pi\big( (n)\big) 
    \chi^\pi\big( (n-a)\cup \rho\big)}{\chi^\pi(1^n)}.
\]
By Lemmas~\ref{LemCharNCycle} and \ref{LemCharHookNmaCycle}, 
the only non-zero terms of the sum correspond to hook partition 
$\pi=(n-i,1^i)$, where $i$ is an integer which fulfils either
$ i\le a-1$ or $i \ge n-a$.

But, for such partitions $\pi$, the dimension $\chi^\pi(1^n)$
is $\binom{n-1}{i}$, and a simple use of the Murnaghan-Nakayama rule shows that
$\chi^\pi( (n))= (-1)^i$.  Another application Murnaghan-Nakayama rule
gives:
\begin{itemize}
    \item for $i \le a-1$, one has that 
$\chi^\pi \big( (n-a) \cup \rho \big)= \chi^{\hat{\pi}}(\rho)$;
where $\hat{\pi}=(a-i,1^i)$ because the only ribbon of size $n-a$
is contained in the first row of $\pi$.
   \item for $i \ge n-a$, one has that
       $\chi^\pi \big( (n-a) \cup \rho \big)= (-1)^{n-a-1} \chi^{\hat{\pi}}(\rho)$
       where $\hat{\pi}=(n-i,1^{i-n+a})$ because the only ribbon of size $n-a$                 
       is contained in the first column of $\pi$.
\end{itemize}
By Lemma ~\ref{LemCharHookExplicit},
$\chi^\pi(\lambda)$ depends on $\eps(\lambda)$ and in a polynomial way on
$m_1(\lambda)$, \dots, $m_{a-1}(\lambda)$.
For each $i \le a-1$ or $i \ge n-a$, the quotient
\[\frac{n!}{n(n-a) \binom{n-1}{i} }\]
is $(n-a-1)!$ multiplied by a polynomial in $n$,
so this completes the proof of Theorem \ref{ThmDependence}.
\qed

\subsection{Small values} \label{SubsectSmallValuesCharacters}
Using Lemma \ref{LemCharHookExplicit}, we can also obtain explicit expressions.
For example,
using the following notation for falling powers
\[ (n)_{m} := n(n-1) \cdots (n-m+1),\]
 for small values of $a$, we get:
\begin{align*}
    \nu_{1^2} (\lambda)&= (1+\eps(\lambda))(n-3)! \frac{n-m_1(\lambda)}{2} ;\\
    \nu_{2} (\lambda) &= (1-\eps(\lambda))(n-3)! \frac{n+m_1(\lambda)-2}{2} ;\\
    \nu_{1^3} (\lambda) &= (1-\eps(\lambda))(n-4)! \frac{1}{6}
    \big( ( n-1)_2 -2(m_1-1)(n-2) +(m_1-1)_2 - 2m_2 \big).
\end{align*}

\section{Bijective approach}\label{SectBijection}
Fix a permutation $\sigma \in \Sym{n}$ of type $\lambda \vdash n$ with the same parity as $a$.
In this 
section, we will prove bijectively Theorem \ref{ThmDependence} in the case $\rho =
1^a$;  that is, up to a factor $(n-a-1)!$, the number of ways to write
$\sigma$ as a product $\alpha \cdot \beta$, where $\alpha$ is a $n$-cycle and $\beta$ a $(n-a)$-cycle
depends polynomially on \[n, m_1(\lambda), \cdots, m_{a-1}(\lambda) \]
and not on the higher multiplicities of $\lambda$.
The proof is a generalisation of the argument given in \cite{CMS}
in the case $a=1$.\medskip

Note that $\alpha=\sigma\cdot \beta^{-1}$ is entirely determined by $\beta$
(recall that $\sigma$ is fixed),
hence the question can be reformulated as follows:
count the number of $(n-a)$-cycles $\beta$ such that $\sigma\cdot \beta^{-1}$
is a full cycle.

\begin{example}\label{ExRunningDef}
    Take the partition $\lambda=(3,2,2)$ of $n=7$ and $a=2$.
    We fix $\sigma=(1\ 2\ 3) (4\ 5) (6\ 7)$. 
    Then the $5$-cycle $\beta=(1\ 3\ 7\ 5\ 2)$ fulfils the condition above.
    Indeed, $\alpha=\sigma\cdot \beta^{-1}=(1\ 3\ 2\ 4\ 5\ 6\ 7)$ is a full
    cycle.
\end{example}

\subsection{A necessary and sufficient condition}
An $(n-a)$-cycle $\beta$ can be written as follows
\begin{equation}
\label{EqBetaB}
\beta=(b_1) (b_2) \cdots (b_{a-1}) (b_a \cdots b_{n-1}), 
\end{equation}
where the $b_i$ are distinct integers between $1$ and $n$.
While it may seem strange to write explicitly $a-1$ fixed
points of $\beta$ in equation \eqref{EqBetaB} instead of all of them,
this is central to our construction.

\begin{example}[Continuing Example \ref{ExRunningDef}]
For $\beta=(1\ 3\ 7\ 5\ 2)$, we can choose $b_1=4$, $b_2=1$, $b_3=3$,
$b_4=7$, $b_5=5$ and $b_6=2$.
This is of course not the only possible choice.
\label{ExRunningBi}
\end{example}

We would like $\alpha=\sigma\cdot \beta^{-1}$ to be a $n$-cycle.
To ensure this condition, we look at the {\em graphical representation}
of $\alpha$.
By definition, it is the directed graph with vertex set $[n]$ and directed edges
$(i,\alpha(i))$ for $i \in [n]$.
This graph is always a union of cycles, corresponding to the cycles of $\alpha$.

\begin{example}\label{ExGraphRep}
The graphical representation of $\sigma=(1\ 2\ 3) (4\ 5) (6\ 7)$
is the left-most graph of Figure \ref{Fig4Graphes}.
\end{example}

The $n-2$ oriented edges 
\begin{equation}\label{EqEdges}
\begin{cases}
b_i \longrightarrow \sigma(b_i) & \text{ for }1 \le i \le a-1 \\
b_i \longrightarrow \sigma(b_{i-1}) & \text{ for }a+1 \le i \le n-1
\end{cases}
\end{equation}
are some of the edges of the graphical representation of permutation $\alpha$.

\begin{example}[Continuing Example \ref{ExRunningBi}]
    \label{ExRunningEdges}
    Choose $b_i$ as above (for $i=1,\dots,6$).
    Then the edges described in 
    \eqref{EqEdges} are drawn on the middle-left graph of Figure 
    \ref{Fig4Graphes}.
They clearly form a subgraph of the graphical representation of
$\alpha=(1\ 3\ 2\ 4\ 5\ 6\ 7)$, which is the right-most graph in Figure
\ref{Fig4Graphes} (for now, disregard the colouring).
\end{example}

The following lemma is an extension of \cite[Proposition 1]{CMS}.

\begin{lemma}\label{LemLehmanSeqToFacto}
Fix an integer $a$ and a permutation $\sigma$ of the same parity as $a$.
Let $b_1$, \dots, $b_{n-1}$ be distinct integers between $1$ and $n$ and
let $\beta$ be the permutation defined by equation~\eqref{EqBetaB}.
Set also $\alpha=\sigma \beta^{-1}$.
Then $\alpha$ is a $n$-cycle if and only if
the edges~\eqref{EqEdges} form an acyclic set of edges.
\end{lemma}
\begin{proof}
    If $\alpha$ is a $n$-cycle, its graphical representation is a cycle
    and, thus, every strict subgraph is acyclic.
    It is in particular the case for the set \eqref{EqEdges} of edges.

In the other direction, let us assume acyclicity of \eqref{EqEdges};  that is, 
the graphical representation of $\alpha$ contains an acyclic subset of edges of size $n-2$.
Hence, it can be either a cycle or the union of two cycles.
But, $\alpha$ has the same sign of a $n$-cycle ($\beta$ is an $(n-a)$-cycle, and
the sign of $\sigma$ corresponds to the parity of $a$) and hence the second possibility
never occurs.
\end{proof}

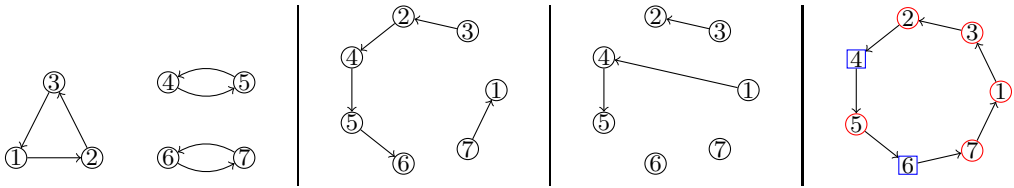
\begin{figure}[t] 
   \[ \begin{array}{c|c|c|c}
       \quad
\begin{tikzpicture}
\tikzstyle{vertex}=[circle,draw,inner sep=0.2pt,minimum size=1mm]
    \node (v1) at (0,0) [vertex] {\footnotesize $1$};
    \node (v2) at (1,0) [vertex] {\footnotesize $2$};
    \node (v3) at (.5,1) [vertex] {\footnotesize $3$};
    \node (v4) at (2,1) [vertex] {\footnotesize $4$};
    \node (v5) at (3,1) [vertex] {\footnotesize $5$};
    \node (v6) at (2,0) [vertex] {\footnotesize $6$};
    \node (v7) at (3,0) [vertex] {\footnotesize $7$};
    \draw [->] (v1) to (v2);
    \draw [->] (v2) to (v3);
    \draw [->] (v3) to (v1);
    \draw [->,bend right=30] (v4) to (v5);
    \draw [->,bend right=30] (v5) to (v4);
    \draw [->,bend right=30] (v6) to (v7);
    \draw [->,bend right=30] (v7) to (v6);
\end{tikzpicture} \quad & \quad
           \begin{tikzpicture}
\tikzstyle{vertex}=[circle,draw,inner sep=0.2pt,minimum size=1mm]
    \node (v1) at (0:1) [vertex] {\footnotesize $1$};
    \node (v2) at (102.8:1) [vertex] {\footnotesize $2$};
    \node (v3) at (51.4:1) [vertex] {\footnotesize $3$};
    \node (v4) at (154.3:1) [vertex] {\footnotesize $4$};
    \node (v5) at (205.7:1) [vertex] {\footnotesize $5$};
    \node (v6) at (257.2:1) [vertex] {\footnotesize $6$};
    \node (v7) at (308.6:1) [vertex] {\footnotesize $7$};
    \draw [->] (v7) to (v1);    
    \draw [->] (v3) to (v2);    
    \draw [->] (v2) to (v4);    
    \draw [->] (v4) to (v5);    
    \draw [->] (v5) to (v6);    
\end{tikzpicture} \quad & \quad
      \begin{tikzpicture}
\tikzstyle{vertex}=[circle,draw,inner sep=0.2pt,minimum size=1mm]
    \node (v1) at (0:1) [vertex] {\footnotesize $1$};
    \node (v2) at (102.8:1) [vertex] {\footnotesize $2$};
    \node (v3) at (51.4:1) [vertex] {\footnotesize $3$};
    \node (v4) at (154.3:1) [vertex] {\footnotesize $4$};
    \node (v5) at (205.7:1) [vertex] {\footnotesize $5$};
    \node (v6) at (257.2:1) [vertex] {\footnotesize $6$};
    \node (v7) at (308.6:1) [vertex] {\footnotesize $7$};
    \draw [->] (v3) to (v2);    
    \draw [->] (v1) to (v4);    
    \draw [->] (v4) to (v5);    
\end{tikzpicture} \quad & \quad
    \begin{tikzpicture}
\tikzstyle{redvertex}=[circle,draw=red,inner sep=0.2pt,minimum size=1mm]
\tikzstyle{bluevertex}=[rectangle,draw=blue,inner sep=0pt,minimum size=2.4mm]
    \node (v1) at (0:1) [redvertex] {\footnotesize $1$};
    \node (v2) at (102.8:1) [redvertex] {\footnotesize $2$};
    \node (v3) at (51.4:1) [redvertex] {\footnotesize $3$};
    \node (v4) at (154.3:1) [bluevertex] {\footnotesize $4$};
    \node (v5) at (205.7:1) [redvertex] {\footnotesize $5$};
    \node (v6) at (257.2:1) [bluevertex] {\footnotesize $6$};
    \node (v7) at (308.6:1) [redvertex] {\footnotesize $7$};
    \draw [->] (v7) to (v1);    
    \draw [->] (v1) to (v3);    
    \draw [->] (v3) to (v2);    
    \draw [->] (v2) to (v4);    
    \draw [->] (v4) to (v5);    
    \draw [->] (v5) to (v6);    
    \draw [->] (v6) to (v7);    
\end{tikzpicture} \quad
\end{array}\]
\caption{Graphs of Examples \ref{ExGraphRep}, \ref{ExRunningEdges},
\ref{ExIllustrationProof} and of the proof of Lemma \ref{LemLehmanKTo1}.}
\label{Fig4Graphes}
\end{figure}

\subsection{Counting sequences}
Let us now try to enumerate sequences $b_1$, \dots, $b_{n-1}$ as above.

We first enumerate sequences $b_1$, \dots, $b_{a-1}$
such that the edges $b_i \to \sigma(b_i)$ do not form any cycle.
Such a sequence is equivalent to the ordered choice, in the graphical representation
of $\sigma$, of $a-1$ edges that do not contain any cycle.

\begin{lemma}
    Let $G$ be a directed graph with $n$ vertices and $n$ edges,
    where $G$ is a disjoint union of $m_i$ cycles of length $i$ for $i \geq 1$.
    Then the number $N_a(\sigma)$ of acyclic subsets of $a-1$ edges is given by
    \begin{equation}\label{EqNbAcyclicSet}
        N_a(\sigma)=\sum_{k=0}^{a-1} \ \sum_{n_1,\dots,n_{a-1} \ge 0 \atop n_1 + \dots+(a-1)n_{a-1}=k}
    (-1)^{n_1+\dots+n_{a-1}} \binom{n-k}{a-1-k} \prod_{i=1}^{a-1} \binom{m_i}{n_i}.
\end{equation}
    In particular, when $a$ is fixed,
    the above equation is a polynomial in $n$, $m_1$, $m_2$, \dots, $m_{a-1}$.
    \label{LemNbSsGraphesAcycliques}
\end{lemma}
\begin{proof}
    Let us label the cycles of the graph by integers from $1$ to $r$
    and denote by $\lambda_j$ the length of cycle $j$.
    Fix a subset $J$ of $[r]$.
    The number of sets of $a-1$ edges containing the cycle labelled $j$
    for all $j$ in $J$ is:
    \[\binom{n-k_J}{a-1-k_J}\text{, where }k_J=\sum_{j \in J} \lambda_j.\]
    We use the convention that the binomial coefficient vanishes if $k_J > a-1$.
    By inclusion-exclusion, the number of sets of $a-1$ edges containing no cycles is:
    \[\sum_{J \subseteq [r]} (-1)^{|J|} \binom{n-k_J}{a-1-k_J}.\]
    We split the sum depending of the value of $k_J$ and on the number $n_i$
    of $j$ in $J$ such that $\lambda_j=i$.
    Doing this, we obtain \eqref{EqNbAcyclicSet} (note that if $n_i >0$ for some $i >a-1$,
    then $k_J >a-1$ and the corresponding term vanishes).
\end{proof}

Finally, the number of possible values for the list $(b_1,\dots,b_{a-1})$
is $(a-1)!N_a(\sigma)$.

The enumeration of the possible values of $b_i$, for $i \ge a$, is harder.
A method due to Lehman is to enumerate such sequences with an additional element $x$
(which is also an integer between $0$ and $n$), with the condition that the edge
$b_a \rightarrow \sigma(x)$ does not create a cycle when added to the graph in \eqref{EqEdges}.

\begin{definition}
    A {\em Lehman sequence of type $a$}
    for $\sigma$ is a sequence $(x,b_1,\dots,b_{n-1})$ 
of integers between $1$ and $n$ such that:
\begin{itemize}
  \item the $b_i$ ($1 \le i \le n-1$) are distinct (but $x$ can be equal to some of the $b_i$);
  \item the following set of $n-1$ edges is acyclic:
  \begin{equation}\label{EqEdgesX}
\begin{cases}
b_i \longrightarrow \sigma(b_i) & \text{ for }1 \le i \le a-1 \\
b_i \longrightarrow \sigma(b_{i-1}) & \text{ for } a < i \le n-1 \\
b_a \longrightarrow \sigma(x).
\end{cases}
\end{equation}
\end{itemize}
\end{definition}

\begin{example}[Continuing Example \ref{ExRunningEdges}]
\label{ExRunningChoixX}
With $\sigma$ and $b_i$ (for $i=1,\cdots,6$) defined as in previous examples,
we can choose any $x \in \{\sigma^{-1}(3),\sigma^{-1}(2),
\sigma^{-1}(4),\sigma^{-1}(5),\sigma^{-1}(6)\}$ so that
$(x,b_1,\cdots,b_6)$ is a Lehman sequence of type $2$ for $\sigma$.
Indeed, the edge $(1,\sigma(x))$ must not create a cycle when added
to the middle-left graph of Figure \ref{Fig4Graphes}.

In general, we shall see later that
the number of possible values for $x$ (when $b_1,\dots,b_{n-1}$ are fixed)
is the oriented distance from $b_a$ to $f$
in the graphical representation of $\alpha$, 
where $f$ is the integer between $1$ and $n$ different from all the $b_i$
(in our example, $b_a=1$ and $f=6$).
\end{example}

With the extra parameter $x$, Lehman sequences are easy to enumerate.
The following is an extension of \cite[Proposition 2]{CMS}).

\begin{lemma}\label{LemLehmanEnumeration}
The number of Lehman sequences of type $a$ for $\sigma$ is
\[n \cdot (a-1)! \cdot N_a(\sigma) \cdot (n-a)!.\] 
As a corollary, for a fixed $a \ge 1$,
it depends polynomially on $n$,$m_1$,\dots,$m_{a-1}$
and not on higher multiplicities.
\end{lemma}

\begin{proof}
We first choose an $(a-1)$-tuple $(b_1,\cdots,b_{a-1})$ such that the
edges $b_i \rightarrow \sigma(b_i)$ do not contain a cycle.
We saw that there are $(a-1)! \cdot N_a(\sigma)$ choices for this list.
Then, choose any number between $1$ and $n$ as the value for $x$.

Let us compute the number of possible values of $b_a$,
after we have fixed the values of $x, b_1, \dots, b_{a-1}$.
By definition of a Lehman sequence, $b_a$ must be different from all values 
$b_1, \dots, b_{a-1}$.
We must also avoid forming a cycle with the edge $b_a \to \sigma(x)$ 
in the graph with edges $b_i \rightarrow \sigma(b_i)$.
This forbids one additional value, the end of the (possibly empty) maximal path beginning
at $\sigma(x)$ in this graph (this maximal path is unique as at most one edge leaves every vertex).
Finally there are $n-a$ possible values for $b_a$.

The same proof shows that, if $b_a,\cdots,b_{j-1}$ are fixed,
there are $n-j$ possible values for $b_j$ (for all $j \in \{a+1,\cdots,n-1\}$).
This ends the proof of the lemma.
\end{proof}

\begin{example}\label{ExIllustrationProof}
    [Illustration of the proof]
    Take as before $\sigma=(1\ 2\ 3)(4\ 5)(6\ 7)$, $b_1=4$, $b_2=1$ and $b_3=3$.
    We also set $x=5$, that is $\sigma(x)=4$.
    As a consequence we know that \eqref{EqEdgesX} contains the middle-right
    graph of Figure \ref{Fig4Graphes}.

We want to find the number of possible values $b_4\neq b_1,b_2,b_3$
such that the edge $b_4 \to \sigma(b_3)=1$ does not add 
any cycle to the configuration.
In addition to the forbidden values $\{4,1,3\}$,
one should also avoid $b_4=5$, which would create a cycle $(1,4,5)$.
So the number of possible values for $b_4$ is $3$,
as asserted in the proof above.
\end{example}

\subsection{From sequences to permutations}
Because of Lemma \ref{LemLehmanSeqToFacto}, to each Lehman sequence we can associate a 
permutation $\beta$ of type $(n-a,1^a)$, such that $\sigma \cdot \beta^{-1}$
is a full cycle.
For fixed $n$ and $\sigma$, denote the function mapping a Lehman sequence of type $a$
for $\sigma$ by $\Lambda_a$.
We now prove that $\Lambda_{a}$ is $k$-to-$1$ (an extension of \cite[Proposition 3]{CMS}), 
for some appropriate number $k$.

\begin{lemma}\label{LemLehmanKTo1}
Let $\beta$ be a permutation of type $(n-a,1^a)$,
such that $\sigma \cdot \beta^{-1}$ is a full cycle.
Set 
\[k=(a-1)! \cdot \frac{a(n-a)}{2} \cdot n\]
Then there are $k$ Lehman sequences of type $a$
that are pre-images of $\beta$ by $\Lambda_a$.
\end{lemma}

\begin{proof}
Fix a permutation $\beta$ as in the statement of the lemma.
A Lehman sequence $(x,b_1,\dots,b_{n-1})$ of type $a$ for $\beta$ is given by
the following:
\begin{itemize}
  \item the choice of which fixed point of $\beta$ does not appear in the
  list $b_1, \dots, b_{a-1}$ (we will denote it $f$);
  \item the choice of an element $b_a$ in the $(n-a)$-cycle of $\beta$
  (the values of $b_{a+1}, b_{a+2}, \dots$ follow from this choice);
  \item the choice of $x$;
  \item a choice of a permutation of the fixed points $b_1, \dots , b_{a-1}$.
\end{itemize}
For the last item, the number of possible permutations is $(a-1)!$ and does not
depend on the choices above.
However, the number of choices for $x$ does depend on $f$ and $b_a$, so we
enumerate directly the number of possible triplets $(x,f,b_a)$.

We consider the graphical representation of 
$\alpha =\sigma \cdot \beta^{-1}$.
By assumption, this is a directed cycle of length $n$.
We can colour its vertices as follows
(denote $G_{\alpha;\beta}$ the resulting coloured graph):
\begin{itemize}
  \item the fixed points of $\beta$
  are vertices coloured in {\em blue}
  (for readers of black-and-white printed versions,
  these are {\em squared} vertices).
  \item the vertices corresponding to points in the support of $\beta$
  are vertices coloured in {\em red}
  ({\em circular} vertices).
\end{itemize}
For example, if $\beta=(1\ 3\ 7\ 5\ 2)$ and $\alpha=(1\ 3\ 2\ 4\ 5\ 6\ 7)$,
then $G_{\alpha;\beta}$ is the right-most graph of Figure \ref{Fig4Graphes}.

With this colouring, $f$ and $b_a$ must be respectively chosen among the
blue and red vertices of $G_{\alpha;\beta}$.
The graph with edges \eqref{EqEdges} is obtained from $G_{\alpha;\beta}$
by erasing the edges leaving the blue vertex $f$
 and the red vertex $b_a$.
 This graph is a disjoint union of two paths, one ending in $b_a$ and
 one ending in $f$ (see Example \ref{ExRunningEdges} where $b_a=1$ and $f=6$).

Now, $x$ must be chosen such that $b_a \to \sigma(x)$ does not create 
a cycle, that is $\sigma(x)$ must be on the path ending in $f$.
Thus, if $b_a$ and $f$ are fixed, there are $d(b_a,f)$ possible values for $x$, 
where $d$ is the {\em oriented} distance from $b_a$ to $f$ in 
$G_{\alpha;\beta}$
(see Example \ref{ExRunningChoixX}).

Finally the number of triplets $(x,f,b_a)$ that yield a Lehman sequence is
\begin{equation}\label{EqChoixDist}
 \sum_{v_b \in V_b \atop v_r \in V_r} d(v_r, v_b),
\end{equation}
where $V_b$ and $V_r$ are respectively the set of blue and red vertices of $G_\alpha$.
Note that $d(v_r, v_b)-1$ is the number of vertices $v$ (of any colour), which are on the
path between $v_r$ and $v_b$ (and different from $v_r$ and $v_b$).
So expression~\eqref{EqChoixDist} can be rewritten as
\[ |V_b| \cdot |V_r| + \left| \left\{ (v_b,v_r,v), \text{ s.t. }\begin{tabular}{c}
$v_b$ is blue ;\\
$v_r$ is red ;\\
$v$ is on the path from $v_r$ to $v_b$.
\end{tabular} \right\} \right| \]
But any set of three distinct vertices of $G_{\alpha;\beta}$
with either two blue and one red vertices 
or two red and one blue vertices can be seen in a unique way as a triplet of the set above.
So finally \eqref{EqChoixDist} is equal to:
\[ |V_b| \cdot |V_r| + \binom{|V_b|}{2} \cdot |V_r| + |V_b| \cdot \binom{|V_r|}{2}
=\frac{|V_b| \cdot |V_r|}{2} (|V_b|+|V_r|). \]
As $|V_b| = a$, while $|V_r|=n-a$, this ends the proof of the lemma.
\end{proof}

Finally, Lemmas~\ref{LemLehmanEnumeration} and \ref{LemLehmanKTo1}
imply the following.
If $\eps(\lambda)=(-1)^{a}$, then the number of $(n-a)$-cycles $\beta$
such that $\sigma \beta^{-1}$ is a cycle is
\begin{equation}
    \nu_{1^a}(\lambda)=\frac{n \cdot (a-1)! \cdot N_a(\sigma) \cdot (n-a)!}
{(a-1)! \cdot \frac{a(n-a)}{2} \cdot n}
=2 \frac{(n-a-1)!}{a} N_a(\sigma). 
\label{EqNuHook}
\end{equation}
The case $\rho=(1^a)$ of Theorem~\ref{ThmDependence}
follows from Lemma \ref{LemNbSsGraphesAcycliques}.

\subsection{Small values}\label{SubsectSmallValuesBijection}

Combining \cref{EqNbAcyclicSet,EqNuHook}, the proof above gives explicit formulae for $\nu_{(1^a)}(\lambda)$, slightly different
from the one obtained with the character approach. For instance,
\begin{align*}
    \nu_{1^2} (\lambda)&= (1+\eps(\lambda))\frac{(n-3)!}{2} (n-m_1(\lambda)) ;\\
    \nu_{1^3} (\lambda) &= (1-\eps(\lambda))\frac{(n-4)!}{3} 
    \left( \binom{n}{2}-(n-1) \cdot m_1+ \binom{m_1}{2} - m_2 \right).
\end{align*}
The expressions look slightly different from the one obtained in
Section \ref{SubsectSmallValuesCharacters}, but they are of course equivalent,
as it can be easily checked by the reader.

We can now give a quick proof of Proposition \ref{PropA2New}, which follows directly
from the above bijective proof. 
If $\sigma$ has no cycles of length smaller than $a$, except $b$ fixed points, then it is clear that $N_a(\sigma)=\binom{n-b}{a-1}$.
This yields the formula claimed in the introduction.

\section{Determining connection coefficients using induction relations}\label{SectInduction}

\subsection{Another proof of Theorem \ref{ThmDependence}}
We now give an induction proof of Theorem \ref{ThmDependence}.  For any
integer partition $\lambda$, let $\lambda^\downi{j}$ be the partition obtained
from $\lambda$ by removing a part of size $j$ and replacing it with a part of
size $j-1$. We do not use this notation if $\lambda$ has no part of size $j$.

\begin{lemma}
    Let $\lambda \vdash n+1$ and $a < n$ a non-negative integer. Then
    \begin{align}
	    (a+1) |C_\lambda| \nu_{1^{a+1}}(\lambda)
     &= (n+1) \sum_{j \geq 2 \atop m_j(\lambda)>0} \left|C_{\lambda^\downi{j}}\right|
     \nu_{1^a}\left(\lambda^{\downi{j}}\right) \cdot (m_{j-1}(\lambda)+1) \cdot 
     (j-1),\label{eq:midproof}
     \\ \textnormal{ and }\ 
     (a+1) \nu_{1^{a+1}}(\lambda) &= \sum_{j \geq
     2 \atop m_j(\lambda)>0}\nu_{1^a}\left(\lambda^{\downi{j}}\right) \cdot m_{j}(\lambda) \cdot j.
     \notag
    \end{align}
    \label{LemAppliedInduction}
\end{lemma}

While we will establish in \cref{sec:cycles} a more general result (Corollary~\ref{thm:corschaeff}),
we elect to
prove Lemma~\ref{LemAppliedInduction} directly immediately below
to make the first four sections self-contained.
\begin{proof}
	Inspecting the left hand side of \eqref{eq:midproof}, we see
	that it counts triples $(\sigma, \alpha,\beta)$ in
	$\Sym{n+1}$ with the following properties: $\sigma \cdot \alpha \cdot
	\beta = e$, where $e$ is the identity element; $\sigma$, $\alpha$, and $\beta$ are of
	types $(n+1)$, $\lambda$, and $(n-a, 1^{a+1})$, respectively;  and $\beta$
	has a distinguished fixed point.  Notice that $\alpha$ and $\beta$ cannot
    share a fixed point because $\sigma$ does not have any,
    and their product
	is $e$.  The right hand side of
	\eqref{eq:midproof} counts the union of
    quadruples $(s, \hats, \hata,\hatb)$ over $\{j \geq 2; m_j(\lambda)>0\}$ with the following properties:
	$s \in [n+1]$; $\hats, \hata$ and $\hatb$ are permutations of
    $[n+1] \setminus \{s\}$ of types $(n)$,
	$\lambda^\downi{j}$ and $(n-a, 1^a)$, respectively with $\hats \cdot  \hata \cdot \hatb=e$;
    and $\hata$ has a member of a
	$(j-1)$-cycle distinguished.  We give a bijection between these two
	sets.

	Suppose that $(\sigma, \alpha, \beta)$ is a triple counted by the left
	hand side \eqref{eq:midproof}, and suppose that $s$ is
	the distinguished fixed point of $\beta$. 
    Denote by $j$ the length of the cycle of $\alpha$ containing $s$.
    Obviously, $m_j(\lambda)>0$ and, besides, $j \ge 2$ as
    $s$ cannot be a fixed point of $\alpha$.
    Simply remove $s$ from
	$\sigma$, $\alpha$ and $\beta$ to obtain permutations $\hats, \hata$ and
	$\hatb$, respectively, and record the image of $s$ in $\alpha$.
	It's not difficult to see that $\hats$, $\hata$ and $\hatb$ have the
	claimed cycle type.  
    Removing $s$ from those permutations is equivalent to
    defining $\hats$, $\hata$ and $\hatb$ as the restrictions of
	$(s\; \sigma(s))\; \sigma$, $(s\; \alpha(s))\;\alpha$ and $\hatb$,
    respectively, to $[n+1] \setminus \{s\}$.  Furthermore, we
	easily compute
	\begin{align*}
		\hats \cdot \hata \cdot \hatb &= (s\; \sigma(s)) \cdot \sigma
		\cdot (s\; \alpha(s)) \cdot \alpha \cdot \beta\\
		&= (s\; \sigma(s)) \cdot (\sigma(s)\;\; \sigma \circ \alpha(s)) \cdot \sigma
		\cdot \alpha \cdot \beta = e,
	\end{align*}
	where the third inequality follows from the fact that $\sigma \circ
	\alpha(s) = s$.

	The construction is clearly reversible, establishing the first assertion of the lemma.  For the second, 
	it is well known that
	\begin{equation*}
		|C_\lambda| = \frac{(n+1)!}{\prod_i i^{m_i(\lambda)} m_i(\lambda)!},
	\end{equation*}
	whence we obtain, for $j \ge 2$ with $m_j(\lambda)>0$,
	\begin{equation*}
        (n+1) \frac{|C_{\lambda^{\downi{j}}}|}{|C_\lambda|} = \frac{j
		\cdot m_j(\lambda)}{(j-1) (m_{j-1}(\lambda) + 1)}.
	\end{equation*}
	Therefore \eqref{eq:midproof} becomes
	\begin{equation*}
		(a+1) \nu_{1^{a+1}}(\lambda) = \sum_{j \geq 2 \atop m_j(\lambda)>0}
		\nu_{1^a}(\lambda^{\downi{j}}) \cdot m_{j}(\lambda) \cdot j,
	\end{equation*}
	completing the proof of the second assertion of the lemma.
\end{proof}

We now prove Theorem \ref{ThmDependence} in the
case $\rho = 1^a$.
\begin{proof}[Proof of Theorem \ref{ThmDependence} for $\rho = 1^a$]
	Our proof proceeds by induction on $a$.  By Theorem \ref{thm:boc}
	(Boccara's Theorem), we have
	\begin{equation*}
		\nu_1(\lambda) = (1 - \eps(\lambda)) (n-2)!,
	\end{equation*}
	establishing the base case.

	Now suppose $a>1$.   By induction,
	there exists a polynomial $G_{a-1}(x_1, \dots, x_{a-1})$ such that for
	any particular partition $\lambda$ with $|\lambda| > a$, we have
	\begin{equation*}
		\nu_{1^{a-1}}(\lambda)  = \frac{\left( 1+(-1)^{a-1}
		\eps(\lambda)\right)(|\lambda|-a)!}{(a-1)!} G_{a-1}(|\lambda|, m_1(\lambda), \dots,
		m_{a-2}(\lambda)).	
	\end{equation*}
	For a fixed $\lambda$, we use Lemma \ref{LemAppliedInduction} to obtain
	that
	\begin{align}
		\nu_{1^a}(\lambda) &= \frac{1}{a} \sum_{j \geq 2 \atop m_j(\lambda)>0} \nu_{1^{a-1}}(\lambda^\downi{j}) \cdot
		m_j(\lambda) \cdot j \notag\\
		&= \frac{1}{a} \sum_{j \geq 2 \atop m_j(\lambda)>0} \tfrac{( 1+(-1)^{a-1} \eps(\lambda^\downi{j})) ( |\lambda^\downi{j}|-a)!}{(a-1)!} G_{a-1}(|\lambda^\downi{j}|,  m_1(\lambda^\downi{j}), \dots,
		m_{a-1}(\lambda^\downi{j})) \cdot m_j(\lambda) \cdot j.\label{eq:needtoexpand}
	\end{align}
	For $j \geq 2$, we see that $\eps(\lambda^\downi{j}) = -\eps(\lambda)$ and
	$|\lambda^\downi{j}| = |\lambda|-1$.  Thus, we see \eqref{eq:needtoexpand} becomes
	\begin{align*}
		& 
		\begin{split}
			\tfrac{(1+ (-1)^a \eps(\lambda)) (|\lambda|-a-1)!}{a!}& \left(\sum_{j \geq 2 \atop m_j(\lambda)>0}^{a-1}
		G_{a-1}(|\lambda|-1, \dots, m_{j-1}(\lambda) + 1, m_j(\lambda) - 1, \dots,
		m_{a-2}(\lambda)) \cdot m_j(\lambda) \cdot j\right.\\
		& + \left. \sum_{j \geq a \atop m_j(\lambda)>0}
		G_{a-1}( |\lambda|-1, m_1(\lambda), \dots,
		m_{a-2}(\lambda)) \cdot m_j(\lambda) \cdot j\right)
		\end{split}
		\\
		& 
		\begin{split}
		=\tfrac{(1+ (-1)^a \eps(\lambda)) (|\lambda|-a-1)!}{a!}&\left(
		\sum_{j \geq 2 \atop m_j(\lambda)>0}^{a-1} G_{a-1}(|\lambda|-1, \dots, m_{j-1}(\lambda) + 1, m_j(\lambda) - 1, \dots,
		m_{a-2}(\lambda)) \cdot m_j(\lambda) \cdot j\right.\\
        & +\left. G_{a-1}(|\lambda| - 1, m_1(\lambda), \dots, m_{a-2}(\lambda))
	\left(|\lambda| - \sum_{j = 1 \atop m_j(\lambda)>0}^{a-1} m_j(\lambda) \cdot j\right) \right).
		\end{split}
	\end{align*}
	Notice that in the two equations immediately above that the condition
	$m_j(\lambda) > 0$ is superfluous because there are no longer terms with
	$\lambda^\downi{j}$, and the terms with $m_j(\lambda)$ are 0.
    The last equation contains a polynomial dependent on $|\lambda|$, and
	$m_i(\lambda)$ for $1 \leq i \leq a-1$, but otherwise independent of $\lambda$,
	completing the proof.
\end{proof}

\subsection{Small values}
Notice that the above proof gives an inductive
formula to compute $\nu_{1^a}(\lambda)$.  For $\lambda \vdash n$, we have
\begin{align*}
	\nu_{1^2}(\lambda) &= \frac{(1+
	\eps(\lambda))(n-3)!}{2}(n-m_1(\lambda));\\
	\nu_{1^3}(\lambda) &= \frac{(1 -\eps(\lambda)) (n-4)!}{3!} ( (n-
	m_1(\lambda))_2  - 2 m_2(\lambda));\\
	\nu_{1^4}(\lambda) &= \frac{(1+\eps(\lambda)) (n-5)!}{4!} \left(\left(n-
	m_1(\lambda)\right)_3  - 6 m_2(\lambda) (n-m_1(\lambda) -2) - 6
	m_3(\lambda)\right).
\end{align*}
Once again, the expressions look different from the ones obtained by the
other method (Sections ~\ref{SubsectSmallValuesCharacters}
and \ref{SubsectSmallValuesBijection}), but they are of course equivalent.

\subsection{A symmetric function formula}\label{sec:symfunc}
We end this section with a symmetric function formula.
This formula will not be used in this paper,
but we mention it 
because it encodes our induction in a very compact way.
As this is not central in the paper, 
we do not recall needed definitions that involve symmetric functions.
We will, however, use the standard notation found in \cite[Chapter I]{Macdonald}.\medskip

Let $\Lambda$ be the symmetric function ring.
It admits a linear basis, called power-sum basis $(p_\lambda)_\lambda$,
indexed by all partitions.  Let $\QQ[\Sym{n}]$ be the group algebra of
$\Sym{n}$.
Let us consider the linear operator
\[ \psi : \begin{array}{rcl} 
	\QQ[\Sym{n}] & \to & \Lambda\\
\sigma & \mapsto & \frac{1}{n!} p_{\lambda},
\end{array}\]
where $\lambda$ is the cycle-type of $\sigma$.
This is a natural transformation, which is used in particular to link
 irreducible characters of the symmetric group with Schur functions
 (see \cite[section I,7]{Macdonald}). Recall from Section \ref{sec:intro} that
 $K_\lambda$ is the sum of permutations of type $\lambda$ and lies in the centre
 of $\QQ[\Sym{n}]$.

The purpose of this section is to give an explicit expression
for $F_a(n):=\psi \big(K_{(n)} \cdot K_{(n-a,1^a)}\big)$ in terms
of monomial symmetric functions $(M_\lambda)$.

From the definitions, we have that
\[F_a(n):=\psi \big(K_{(n)} \cdot K_{(n-a,1^a)}\big) = \frac{1}{n!}
\sum_{\lambda \vdash n} |C_\lambda| \nu_{1^a}(\lambda) p_\lambda. \]

As always in this paper, the case $a=1$ is easy.
For $a=1$,
one has $\nu_1(\lambda)=(1-\eps(\lambda))(n-2)!$. 
Denoting as usual $z_\lambda=n!/|C_\lambda|$, one has
\begin{multline}
F_1(n) =  (n-2)! \sum_{\lambda \vdash n} (1-\eps(\lambda)) \frac{p_\lambda}{z_\lambda}
=  (n-2)! \left( \sum_{\lambda \vdash n}\frac{p_\lambda}{z_\lambda}
- \sum_{\lambda \vdash n} \eps(\lambda) \frac{p_\lambda}{z_\lambda}\right)
 = (n-2)! \sum_{\lambda \vdash n \atop \lambda \neq (1^n)} M_\lambda,
 \label{EqF1}
\end{multline}
where the last equality comes from \cite[Chapter I, Eq (2.14')]{Macdonald}.

We then proceed by induction using to the following lemma. 
 \begin{lemma}\label{LemIndFa}
 For $a\geq 0$,
 \[ F_{a+1}(n+1) = \frac{1}{a+1} \Delta\big(F_a(n) \big), \]
 where $\Delta$ is the differential operator
 \[ \sum_i x_i^2 \frac{\partial}{\partial x_i}. \]
\end{lemma}
\begin{proof}
Fix $a \ge 0$. On the one hand, one has that
\begin{align*}
    F_{a+1}(n+1)& = \frac{1}{(n+1)!} \sum_{\mu \vdash n+1} 
    |C_\mu| \nu_{1^{a+1}}(\mu) p_\mu \\
&= \frac{n+1}{(n+1)! \cdot (a+1)} \sum_{\mu \vdash n+1} p_\mu
\left( \sum_{j \geq 2 \atop m_j(\mu) >0} \left|C_{\mu^\downi{j}}\right|
\nu_{1^a}\left(\mu^{\downi{j}}\right) \cdot (m_{j-1}(\mu)+1) \cdot (j-1) \right),
  \end{align*}
  where the second equality comes from Lemma \ref{LemAppliedInduction}.

  Let $\lambda^{\uparrow (j)}$
the partition obtained from $\lambda$
by replacing a part of size $j$ with a part of size $j+1$
(again, the notation is defined only if $m_j(\lambda)>0$).
  The function $(\mu,j) \mapsto (\mu^{\downarrow (j)},j)$ is
  clearly bijective from the set $\{(\mu,j), \ \mu \vdash n+1,\ j \ge 2,\ m_j(\mu)>0\}$
  to $\{(\lambda,j), \ \lambda \vdash n,\ j \ge 2,\ m_{j-1}(\lambda)>0\}$
-- the inverse is $(\lambda,j) \mapsto (\lambda^{\uparrow (j-1)},j)$ --
so we can change the summation indices.
\begin{align*}
    F_{a+1}(n+1) &=\frac{1}{n! \cdot (a+1)} \sum_{\lambda \vdash n}
    \left|C_\lambda\right| \nu_{1^a}\left(\lambda\right)
    \left(\sum_{j \geq 2 \atop m_{j-1}(\lambda)>0} \big(m_{j-1}({\lambda^{\uparrow (j-1)}})+1 \big) \cdot (j-1) 
\cdot p_{\lambda^{\uparrow (j-1)}} \right)\\
&= \frac{1}{n! \cdot (a+1)} \sum_{\lambda \vdash n}                         
\left|C_\lambda\right| \nu_{1^a}\left(\lambda\right) 
\left(\sum_{j \geq 2 \atop m_{j-1}(\lambda)>0} m_{j-1}(\lambda) \cdot (j-1) 
\cdot p_{\lambda^{\uparrow (j-1)}} \right).
\end{align*}
But, on the other hand, it is easy to check that
\[ \sum_{j \geq 2 \atop m_{j-1}(\lambda) >0} m_{j-1}(\lambda) \cdot (j-1) 
\cdot p_{\lambda^{\uparrow (j-1)}} = \Delta(p_\lambda).
\]
A detailed proof of the previous equation can be found in 
\cite[Section 6.1]{MoiEtKatyaBijection};
the reader uncomfortable with infinitely many variables
should also look at this reference, where this issue is discussed.
Finally, using the linearity of the operator $\Delta$, we get that
\[F_{a+1}(n+1) = \frac{1}{a+1} \Delta \left( \frac{1}{n!}
\sum_{\lambda \vdash n} \left|C_\lambda\right| \nu_{1^a}\left(\lambda\right) p_\lambda\right), \]
finishing the proof.
\end{proof}

We need one more lemma before proceeding to the main theorem of this section.

\begin{lemma}
    The operator $\Delta$ is injective.
    \label{LemmeDeltaInj}
\end{lemma}
\begin{proof}
    As $\Delta$ is homogeneous 
    (it sends symmetric functions of degree $n$ to symmetric functions of degree $n+1$),
    it is enough to prove the lemma for the restriction to symmetric functions of a specific degree $n$.
    Consider a linear combination $\sum_{\lambda \vdash n} a_\lambda M_\lambda$
    sent by $\Delta$ onto $0$.
    Using the description of the action of $\Delta$ on the monomial basis
    given in \cite[Section 6.1]{MoiEtKatyaBijection}, we get
    \[ \sum_{\lambda \vdash n} a_\lambda 
    \left( \sum_{j>0 \atop m_j(\lambda) >0} j \cdot (m_{j+1}(\lambda)+1) \cdot M_{\lambda^{\uparrow (j)}} \right)=0.\]
    Rewriting this as usual as a sum over $\mu$ and $j$, we get
    \[ \sum_{\mu \vdash n+1} M_\mu \left( \sum_{j>0 \atop m_{j+1}(\mu) >0} a_{\mu^{\downarrow (j+1)}} \cdot j \cdot
    m_{j+1}(\mu) \right) =0.\]
    This equation is equivalent to the following linear system:
    for any $\mu \vdash n+1$,
    \begin{equation}
        \sum_{j>0 \atop m_{j+1}(\mu)>0} a_{\mu^{\downarrow (j+1)}} \cdot j \cdot
            m_{j+1}(\mu) =0
            \tag{$E_\mu$}
        \label{EqEMu}
    \end{equation}
    Fix $\lambda \vdash n$ and $\mu=(\lambda_1+1,\lambda_2,\lambda_3,\dots)$.
    Then equation \eqref{EqEMu} involves the variable $a_\lambda$ and variables $a_\nu$
    with $\nu \vdash n$ and $\nu_1=\lambda_1+1$.
    In particular, such partitions $\nu$ are lexicographically bigger than $\lambda$.

    In other terms, the linear system $\eqref{EqEMu}_{\mu \vdash n+1}$ admits a triangular
    subsystem and the only solution is $a_\lambda=0$ for all $\lambda \vdash n$.
\end{proof}

We can now obtain an explicit formula for $F_a(n)$.
By convention, we set $(x)_{-1}=\frac{1}{x+1}$, which is compatible with the relation
$(n)_m \cdot (n-m) = (n)_{m+1}$.
\begin{theorem}
    For any $n>0$ and $a\ge 0$, one has that
    \[F_a(n)= \frac{(n-a-1)!}{a!} \sum_{\lambda \vdash n \atop |\lambda|-\ell(\lambda) \ge a}
    (|\lambda|-\ell(\lambda))_{a-1} M_\lambda. \]
    \label{ThmSym}
\end{theorem}
\begin{proof}
Let us first look at the case $a \ge 1$.
We proceed by induction.
For $a=1$, the theorem corresponds to equation \eqref{EqF1}.
Let us suppose that the theorem is true for a fixed $a\ge 1$.
Then, by Lemma \ref{LemIndFa} and induction hypothesis, one has:
\[F_{a+1}(n+1) = \frac{1}{a+1} \Delta\big(F_a(n) \big)
= \frac{(n-a-1)!}{(a+1)!} 
\sum_{\lambda \vdash n \atop |\lambda|-\ell(\lambda) \ge a}
    (|\lambda|-\ell(\lambda))_{a-1}  \Delta(M_\lambda).\]
But, as explained in 
\cite[Section 6.1]{MoiEtKatyaBijection},

\[\Delta(M_\lambda) = \sum_{j>0 \atop m_j(\lambda) > 0} j \cdot (m_{j+1}(\lambda)+1) \cdot M_{\lambda^{\uparrow (j)}}.\]
Substituting this into the equation above, we have that
\[F_{a+1}(n+1) = \frac{(n-a-1)!}{(a+1)!} \sum_{\lambda \vdash n \atop |\lambda|-\ell(\lambda) \ge a}
\sum_{j>0 \atop m_j(\lambda) > 0} (|\lambda|-\ell(\lambda))_{a-1} \cdot j \cdot (m_{j+1}(\lambda)+1) \cdot M_{\lambda^{\uparrow (j)}}. \]
We use the same manipulation as in the proof of Lemma \ref{LemIndFa}:
the double sum on $\lambda \vdash n$ and $j$ can be turned into a double sum on 
$\mu \vdash n+1$ and $j$ where $\mu=\lambda^{\uparrow (j)}$.
Note that the size and length of $\mu$ are related to $\lambda$ by
$|\mu|=|\lambda|+1$ and $\ell(\mu)=\ell(\lambda)$.
We get that
\begin{align*}
    F_{a+1}(n+1) &= \frac{(n-a-1)!}{(a+1)!} \sum_{\mu \vdash n+1 \atop |\mu|-\ell(\mu) \ge a+1}
\sum_{j>0} (|\mu|-\ell(\mu)-1)_{a-1} \cdot j \cdot (m_{j+1}(\mu)) \cdot M_\mu \\
&= \frac{(n-a-1)!}{(a+1)!} \sum_{\mu \vdash n+1 \atop |\mu|-\ell(\mu) \ge a+1}
M_\mu \cdot (|\mu|-\ell(\mu)-1)_{a-1} \left( \sum_{j>0} j \cdot (m_{j+1}(\mu)) \right) \\
&= \frac{(n-a-1)!}{(a+1)!} \sum_{\mu \vdash n+1 \atop |\mu|-\ell(\mu) \ge a+1}
M_\mu \cdot (|\mu|-\ell(\mu))_a.
\end{align*}
This ends the proof of the case $a \ge 1$.

For $a = 0$, the same computation as above shows that
\[\Delta \left(  (n-1)! \sum_{\lambda \vdash n} \frac{1}{|\lambda|-\ell(\lambda)+1} 
M_\lambda \right) = (n-1)! \sum_{\mu \vdash n+1 \atop |\mu|-\ell(\mu) \ge 1} M_\mu
=F_1(n+1).\]
But, by Lemma \ref{LemIndFa}, one also has $\Delta(F_0(n))=F_1(n+1)$.
Hence the theorem follows from the injectivity of $\Delta$, given in Lemma
\ref{LemmeDeltaInj}.
\end{proof}

\begin{remark}
    The case $a=0$ of Theorem \ref{ThmSym} was already known,
    under a slightly different form; see
    \cite[Theorem 2]{VassilievaExplicitMonomialExpansions}.
\end{remark}    
\begin{remark}
	Theorem 1 of \cite{VassilievaExplicitMonomialExpansions} also contains a symmetric function involving all connection coefficients $c_{\mu,\lambda}^{(n)}$;
    We can show our symmetric function formula contained in 
    Theorem \ref{ThmSym} can be obtained from
    \cite[Theorem 1]{VassilievaExplicitMonomialExpansions}
    {\em via} involved $\lambda$-ring manipulations.  The presentation above,
    however, is just as short and doesn't rely on the non-trivial result
    \cite[Theorem 1]{VassilievaExplicitMonomialExpansions}.
\end{remark}

\section{Separation Probabilities}\label{SectSeparability}

\begin{definition}[From \cite{BernardiEtAlSeparation}]
    Let $J$ be a set partition $(J_1, \ldots, J_k)$ of $[m]$ of length $k$.
    A permutation $\sigma \in \Sym{n}$ (with $n \geq m$) is called $J$-separated
    if no cycle of $\sigma$ contains two elements from different blocks of $J$.  
    It is said to be {\em strongly $J$-separated} if, moreover,
    each $J_i$ is contained in some cycle of $\sigma$.
\end{definition}

It is easy to see that highly symmetric sets (\emph{i.e.} sets of
permutations whose sum lie in the centre of the group algebra of $\Sym{n}$), that the number of $J$-separated permutations only depends on the composition $I = (i_1, \dots, i_k)$ of $m$,
where $i_p = \# J_p$.  By convention, a permutation is $I$-separated
for a composition $I$, if it is $(J_1, \dots, J_k)$-separated, where $J_p =
\{i_{p-1} + 1, \dots, i_p\}$.  We use the two notions of separation given by a set
partition and separation given by a composition interchangeably.

The same convention holds for strongly separable permutations.\smallskip

As mentioned at the end of the paper \cite{BernardiEtAlSeparation},
the problem of computing separation probabilities or strong separation
probabilities
are linked to each other by a simple relation (equation (34) of the cited paper).
Therefore, we shall use the notion depending on which one is most convenient.
Note also, that for $I=(1^k)$ (the case for which we have explicit expressions),
the two notions coincide.

The rest of Section \ref{SectSeparability} is devoted to computing the strong separation probability of the product of a random $n$-cycle with a random $(n-a)$-cycle, whereas in
Section \ref{sec:nonfullsep} we use the standard notion of separable permutations.
Therefore when we talk about separable permutation in the remainder of Section
\ref{SectSeparability}, we will mean strongly separable.

An induction formula for these quantities is given
in 
Section \ref{SubSectSepProd}.
To do that, we need some preliminary results regarding
the computation of separation probabilities in simpler models:
uniform random permutations 
(Section \ref{SubSectSepUnif}),
uniform random odd permutations 
(Section \ref{SubSectSepOdd}) and
uniform random permutations of a given type 
(Section \ref{SubSectSepType}).

\subsection{Uniform random permutation}\label{SubSectSepUnif}
Let $I = (i_1, \dots, i_k)$ be a composition of $m$ of length $k$.
The following result can be found in \cite[p. 13]{StanleyTalkPP2010},
but we copy it here for completeness.

The probability $P^I$ that a uniform random permutation in $\Sym{n}$ ($n \geq m$)
is strongly $I$-separated is given by
\begin{equation}
    P^I = \frac{(i_1-1)!\cdot(i_2-1)! \cdots (i_k-1)!}{m!}.
    \label{eq:PI}
\end{equation}
\begin{proof}
The result is obvious for $n=m$.
Indeed, if $J$ is a set-partition of $[m]$,
a permutation in $\Sym{m}$ is $J$-separated if
it is the disjoint product of a cycle of support $J_1$ ($(i_1-1)!$ choices),
a cycle of support $J_2$ ($(i_2-1)!$ choices) and so on.

We then proceed by induction on $n$.
A random permutation of $n+1$ can be obtained from a random permutation $\sigma$
of $n$ by choosing uniformly an integer $j$ between $1$ and $n+1$ and
\begin{itemize}
    \item either add $n+1$ as a fixed point if $j=n+1$;
    \item or add $n+1$ right after $j$ in the cycle notation of $\sigma$.
\end{itemize}
Both operations do not change the fact that the permutation is $J$-separated,
so the probability of being $J$ separated does not depend on $n$.
\end{proof}

\subsection{Uniform random odd permutation}\label{SubSectSepOdd}
Let $I$ be a composition of $m$ of length $k$.                     
The probabilities $P^I_{n,\text{odd}}$ and $P^I_{n,\text{even}}$
for a uniform random odd (respectively even) permutation in $\Sym{n}$
($n > m$) to be $I$-separated fulfils the relation:
\begin{equation}
     P^I_{n,\text{odd}} = \frac{n-1}{n} (P^I_{n-1,\text{even}})
        + \frac{1}{n} P^I_{n-1,\text{odd}}. 
    \label{EqProbaOddPermInd}
\end{equation}
\begin{proof}
    Let $\Odd_n$ and $\Even_n$ be the set of odd and even permutation in
    $\Sym{n}$, respectively.  Then, one has a bijection
    \[\phi: \Odd_n \simeq \Odd_{n-1} \cup ([n-1] \times \Even_{n-1}), \]
    where $\phi$ and its inverse are given by the following rules.
    \begin{itemize}
        \item For $\sigma \in \Sym{n}$, set $\phi(\sigma)$ to be the restriction of $\sigma$ to $[n-1]$ if $n$ is a fixed point of
            $\sigma$; otherwise, let $j= \sigma^{-1}(n) \neq n$
            and set $\phi(\sigma)=(j,\tau)$,
            where $\tau$ is obtained by erasing $n$
            in the cycle decomposition of $\sigma$.
    \item $\phi^{-1}(\tau)$ is its canonical embedding in $\Sym{n}$ if $\tau$
            is an odd permutation of $[n-1]$;
            otherwise
            $\phi^{-1}(j,\tau)=\tau (j\ n)$
            for $(j,\tau) \in [n-1] \times \Even_{n-1}$.
    \end{itemize}
    Clearly, in both cases, $\sigma$ is $I$-separated if and only if $\tau$ is.
    A quick computation leads to the formula above.
\end{proof}
As a random permutation in $\Sym{n}$ has probability $P^I$ to be $I$-separated,
one has:
\[P^I_{n,\text{odd}} + P^I_{n,\text{even}} = 2P^I.\]
Thus \eqref{EqProbaOddPermInd} can be made into the induction relation:
\[P^I_{n,\text{odd}} = \frac{n-1}{n} (2P^I - P^I_{n-1,\text{odd}})       
        + \frac{1}{n} P^I_{n-1,\text{odd}}.\]
To solve this induction, let us define
$Q^I_n = (P^I_{n,\text{odd}} - P^I)$.
Our induction then becomes
\[ Q^I_n = - \frac{n-2}{n} Q^I_{n-1} \]
Together with the base case (the argument to compute $P^I_{m,\text{odd}}$
is the same than in the proof of \eqref{eq:PI})
\[Q^I_m = \begin{cases}
    0 - P^I \text{ if } m-k \text{ is even,}\\
    2 P^I -P^I\text{ if } m-k \text{ is odd,}
\end{cases}\]
we obtain that
\[Q^I_n=(-1)^{n-k+1} P^I \frac{m(m-1)}{n(n-1)}.\]
Therefore, we have proved the following proposition (this result is also given in \cite[p. 61]{StanleyTalkPP2010}). 
\begin{proposition}
    Let $I$ be a composition of $m$ of length $k$ and $n$ an integer with $n \ge m$.
	The probability for a random odd permutation in $\Sym{n}$ to be $I$-separated is
\[P^I_{n,\text{odd}} = P^I \cdot
    \left( 1 + (-1)^{n-k+1} \frac{m(m-1)}{n(n-1)} \right).\]
 \end{proposition}

\subsection{Uniform random permutation of a given type $\lambda$}
 \label{SubSectSepType}
 \begin{proposition}
    Let $I=(i_1,\dots,i_k)$ be a composition of $m$ of length $k$
    and $\lambda$ a partition of $n$ of length $r$.
    The probability that a random permutation of type $\lambda$ is
    $I$-separated is given by
    \begin{equation}
        P^I(\lambda)=
        \frac{1}{(n)_{m}} 
        \sum_{1 \le j_1,\dots, j_k \le r \atop \text{distinct}}
        \prod_{t=1}^k (\lambda_{j_t})_{{i_t}}.
        \label{EqProbaTypeLambda}
    \end{equation}
 \end{proposition}
 \begin{proof}
     Each partition $\pi$ of type $\lambda$
     can be written in exactly $z_\lambda:=n!/|C_\lambda|$ ways as
     \[\pi=\big(b_1\ \dots\ b_{\lambda_1}\big) 
     \big(b_{\lambda_1+1}\ \dots\ b_{\lambda_1+\lambda_2}\big) \dots,\]
     with $b_1,\dots,b_n$ containing exactly once each number from $1$ to $n$.
     We denote $C_j(\lambda)$ the set of indices in the $j$-th cycle of $\pi$,
     that is
     \[C_j(\lambda) :=\big\{\lambda_1+\dots+\lambda_{j-1}+1,
     \dots,\lambda_1+\dots+\lambda_{j} \big\}.\]

     Fix some distinct integers $j_1,\dots, j_k$ between $1$ and $r$.
     As in the introduction of the section, we denote
     $J_h = \{i_{h-1} + 1, \dots, i_h\}$.
     Let us consider the property: ``for each $h$, the numbers in $J_h$ lie in the
     $j_h$-th cycle of $\pi$''.

     For this property to be true, each $b^{-1}(J_h)$ must be included in
     $C_h(\lambda)$, which means that there are
     $(\lambda_{j_h})_{i_h}$ possible values for the list
     $(b^{-1}(i_{h-1} + 1),\dots,b^{-1}(i_h))$.
     Thus, in total,
     there are $\prod_{t=1}^k (\lambda_{j_t})_{i_t}$ possible values
     for $(b^{-1}(1),\dots,b^{-1}(m))$ such that the considered property holds.
     As there are in total $(n)_m$ possible values for this list
     and as this list is uniformly distributed when we take $\pi$ uniformly at random in $C_\lambda$,
     the probability that our property holds is
     \[\frac{1}{(n)_{m}} \prod_{t=1}^k (\lambda_{j_t})_{i_t}. \]

     For different sequences $j_1,\dots, j_k$ these events are incompatible
     and their union correspond to the fact that $\pi$ is 
     strongly $J$-separated, whence the result follows.
 \end{proof}
 Let us denote $R_I(\lambda)=(n)_m P^I(\lambda)$
 the sum in equation \eqref{EqProbaTypeLambda}
 (it is a polynomial function in $\lambda_1,\lambda_2,\dots$).
 The following lemma will be useful in the next section.
 \begin{lemma}
     Let $I=(i_1,\dots,i_k)$ be a composition of $m$ of length $k$.
     For any partition $\lambda$ of $n$ of length $r$, one has that
     \[\sum_{t=1}^{r}
     \lambda_t R_I(\lambda + \delta_t)
     = (n+m) R_I(\lambda) + \sum_{h=1}^k i_h (i_h - 1)
     R_{I-\delta_h}(\lambda),\]
     where $\delta_t$ is the vector with only $0$ components, except for
     its $t$-th component which equals $1$.
     \label{LemRecRI}
 \end{lemma}
 \begin{proof}
     Let us fix $t \in [r]$ and
     compare $R_I(\lambda + \delta_t)$ and $R_I(\lambda)$.
     The distinct summation indices $j_1,\dots, j_k$ obtained from
     \eqref{EqProbaTypeLambda}  in $R_I(\lambda + \delta_t)$ and $R_I(\lambda)$ are the same
     and the summands are different only in the case where one of the $j_s$,
     let us say $j_h$, is equal to $t$. Thus, one has that
     \[
         R_I(\lambda + \delta_t) - R_I(\lambda)
         = \sum_{h=1}^k \sum_{j_1,\dots,\widehat{j_h},\dots,j_k \neq t}
      \big( (\lambda_t+1)_{i_h} - (\lambda_t)_{i_h} \big)
      \prod_{g \neq h} (\lambda_{j_g})_{i_g}.
      \]
      The hat in the summation index means that we sum over choices
      of integers $j_g$ for $g \neq h$. 
      But, by definition:
      \[ \sum_{j_1,\dots,\widehat{j_h},\dots,j_k \neq t}
               \prod_{g \neq h} (\lambda_{j_g})_{i_g}
               = R_{I \backslash i_h}(\lambda \backslash \lambda_t). \]
               Thus, noticing that $(\lambda_t+1)_{i_h} - (\lambda_t)_{i_h}=i_h (\lambda_t)_{(i_h -1)}$,
               we obtain:
      \[ R_I(\lambda + \delta_t) - R_I(\lambda) =
        \sum_{h=1}^k i_h (\lambda_t)_{(i_h -1)}
        R_{I \backslash i_h}(\lambda \backslash \lambda_t). \]
        Now we multiply by $\lambda_t$ and sum this equality over $t$:
        \begin{multline*}
            \sum_{t=1}^r \big( \lambda_t R_I(\lambda + \delta_t) \big) 
                                    - n \cdot R_I(\lambda) =
           \sum_{t=1}^r \sum_{h=1}^k i_h \lambda_t (\lambda_t)_{(i_h -1)}
           R_{I \backslash i_h}(\lambda \backslash \lambda_t)\\
           =\sum_{h=1}^k \left( i_h \sum_{t=1}^r \big((\lambda_t)_{i_h}
           + (i_h-1) (\lambda_t)_{(i_h -1)}\big) 
           R_{I \backslash i_h}(\lambda \backslash \lambda_t)\right).
        \end{multline*}
        It is clear from the definition of $R_I$ that these functions satisfy
        \[ \sum_{t=1}^r (\lambda_t)_{i} R_I(\lambda \backslash \lambda_t)
            = R_{I \cup i}(\lambda). \]
        Indeed, we simply split the sum in $R_{I \cup i}(\lambda)$
        depending on which part of $\lambda$ is associated to the part $i$
        of $I \cup i$.
        Returning to our computation, one has that
        \[\sum_{t=1}^r \big( \lambda_t R_I(\lambda + \delta_t) \big)
            - n \cdot R_I(\lambda) =
        \sum_{h=1}^k \left( i_h R_I(\lambda) + 
        i_h (i_h-1) R_{I \backslash i_h \cup (i_h -1)}(\lambda) \right), \]
	completing the proof.
 \end{proof}

 \subsection{Product of a uniform random $n$-cycle and a uniform random $(n-a)$-cycle.}\label{sec:defsigprob}
 \label{SubSectSepProd}

Define $\pi^I_{\lambda, \mu}$ to be the probability that
the product of two uniformly selected random permutations
of types $\lambda$ and $\mu$, respectively, is strongly $I$-separated.
For ease of notation, let $P^I_{n,a} := \pi^I_{(n), (n-a, 1^a)}$. 

The case $a=1$ is particularly simple.
Indeed, using Theorem \ref{thm:boc},
this is equivalent to choosing a odd
permutation uniformly at random.
Therefore, the computation of 
Section \ref{SubSectSepOdd} 
gives that
\begin{equation}\label{EqPn1}
    P^I_{n,1} = P^I_{n,\text{odd}} =
    P^I \cdot \left( 1 + (-1)^{n-k+1} \frac{m(m-1)}{n(n-1)} \right).
\end{equation}

We will then establish an induction relation for $P^I_{n,a}$.
Let us begin with an elementary lemma.
\begin{lemma}\label{LemProbCond}
    Let $I$ be a composition of an integer $m$ and $n$ and 
    $a$ be non-negative integers with $n>a$ and $n \ge m$.
    \[P^I_{n,a} = \frac{a!(n-a)}{(n-1)! n!}
    \sum_{\lambda \vdash n} |C_\lambda| \nu_{1^a}(\lambda) P^I(\lambda). \]
\end{lemma}

\begin{proof}
    The probability that the product $\alpha \cdot \beta$, where $\alpha$ and
    $\beta$ are an $n$-cycle and $(n-a)$-cycle respectively, has cycle type $\lambda$
    is given by 
    \[ \frac{|C_\lambda| \nu_{1^a}(\lambda)}{(n-1)! \frac{n!}{a!(n-a)}} \]
    Of course, all permutations of a given type $\lambda$ have the same
    probability to occur.
    Therefore, the conditional probability that $\alpha \beta$ is $I$-separated
    knowing that it has cycle type $\lambda$ is $P^I(\lambda)$.
    Putting everything together, we obtain the formula above.
\end{proof}

\begin{theorem}
    Let $I=(i_1,\dots,i_k)$ be a composition of an integer $m$ and $n$ and 
    $a$ be non-negative integers with $n \ge m$.     One has that
    \[ n(n+1)P^I_{n+1,a+1} =
    (n-m+1)(n+m)P^I_{n,a} + \sum_{h=1}^k i_h(i_h-1) P^{I-\delta_h}_{n,a}. \]
    \label{ThmRecBonaPb}
\end{theorem}
\begin{proof}
    By lemma \ref{LemProbCond}, one has that
    \[P^I_{n+1,a+1} = \frac{(a+1)!(n-a)}{n! (n+1)!}
    \sum_{\mu \vdash n+1} |C_\mu| \nu_{1^{a+1}}(\mu) P^I(\mu). \]
    We now use Lemma \ref{LemAppliedInduction} and obtain:
    \[
        P^I_{n+1,a+1} = \frac{(a+1)!(n-a)}{n! (n+1)!}
        \sum_{\mu \vdash n+1} \frac{n+1}{a+1}
    \sum_{j \geq 2 \atop m_j(\mu)>0}
    |C_{\mu^{\downarrow (j)}}| \nu_{1^{a}}\big(\mu^{\downarrow (j)}\big) (j-1) (m_{j-1}(\mu)+1)  P^I(\mu)
    \]    
    We use the same reasoning
    as in the proof of \cref{LemIndFa}:
        the function $(\mu,j) \mapsto (\mu^{\downarrow (j)},j)$ is
        bijective 
        -- the inverse is $(\lambda,j) \mapsto (\lambda^{\uparrow (j-1)},j)$ --
        so we can change the summation indices.
        \begin{align*}
        P^I_{n+1,a+1} &= \frac{a!(n-a)}{n! n!}
        \sum_{\lambda \vdash n} \sum_{j \geq 2 \atop m_{j-1}(\lambda)>0}
    (j-1) m_{j-1}(\lambda) |C_\lambda| \nu_{1^a}(\lambda) P^I(\lambda^{\uparrow (j-1)}) \\
    &= \frac{a!(n-a)}{n! n!} \sum_{\lambda \vdash n} |C_\lambda| \nu_{1^a}(\lambda)
            \sum_{t=1}^{\ell(\lambda)} \lambda_t P^I(\lambda + \delta_t).
    \end{align*}
    The last inequality comes from the fact that, for a fixed partition $\lambda$ if we consider the multiset formed
    by $(j-1) m_{j-1}(\lambda)$ copies of $\lambda^{\uparrow (j-1)}$ for each $j$ with $m_{j-1}(\lambda)>0$,
    it is the same as the multiset with $\lambda_t$ copies of $\lambda + \delta_t$ for each $t $ between $1$
    and $\ell(\lambda)$.
     
    Now, notice that we have already computed in Lemma \ref{LemRecRI} the sum over $t$ appearing
     in the previous equation
    (recall that $P^I(\lambda)=R_I(\lambda)/(n)_{m}$).
    We obtain:
  \[
        P^I_{n+1,a+1} = \frac{a!(n-a)}{n! n!}
	\sum_{\lambda \vdash n} |C_\lambda| \nu_{1^a}(\lambda) \cdot \\
        \left( \frac{(n+m)(n-m+1)}{n+1} P^I(\lambda) +
        \sum_{h=1}^k \frac{i_h(i_h-1)}{n+1} P^{I-\delta_h}(\lambda) \right).
  \]
    We can now use Lemma \ref{LemProbCond} again and we get that
    \[ P^I_{n+1,a+1} = \frac{(n+m)(n-m+1)}{n(n+1)} P^I_{n,a} + \frac{1}{n(n+1)}
    \left( \sum_{h=1}^k i_h(i_h-1) P^{I-\delta_h}_{n,a} \right). \qedhere\]
\end{proof}

Recall that $P_{n,1}^I$ is given by the simple formula \eqref{EqPn1}.
Using this,
Theorem~\ref{ThmRecBonaPb} allows us to compute $P_{n,a}^I$ for $a \ge 1$
by induction on $a$
and $P_{n,0}^I$ by induction on $|I|-\ell(I)$.
We end this section by giving a few applications of this fact.

The induction on $a$ for $a \ge 1$
becomes particularly simple in the case $I=(1^k)$.
In this case, $m=k$ and the sum in Theorem~\ref{ThmRecBonaPb} vanishes. We
therefore find that
\[ n(n+1)P^{(1^k)}_{n+1,a+1} =
(n-k+1)(n+k)P^{(1^k)}_{n,a},\]
from which we deduce immediately the following formula.
\begin{proposition}
    \label{Prop1KSep}
    Let $n$, $a$ and $k$ be positive integers with $a,k \le n$.
    Then
    \[P_{n,a}^{(1^k)} = \frac{(n-1+k)_{a-1} \cdot (n-k)_{a-1}}
    {(n)_{a-1} \cdot (n-a)_{a-1}}
\left( 1 + (-1)^{n-k} \frac{k(k-1)}{n(n+1)} \right).\]
\end{proposition}

But, Theorem \ref{ThmRecBonaPb} is also valid for $a=0$ and hence
\begin{equation}      (n-m+1)(n+m)P^I_{n,0}= n(n+1)P^I_{n+1,1}
\label{EqIndA0}
- \sum_{h=1}^k i_h(i_h-1) P^{I-\delta_h}_{n,0}. \end{equation}
Again, the case $I=(1^k)$ is particularly easy as the sum vanishes: we get that
\begin{equation*}
    (n-k+1)(n+k)P^{(1^k)}_{n,0}=n(n+1) \frac{1}{k!}
\left( 1 + (-1)^{n-k} \frac{k(k-1)}{n(n+1)} \right),
\end{equation*}
which simplifies to
\begin{equation}
    P^{(1^k)}_{n,0} =\begin{cases}
    \frac{1}{k!} &\text{ if $n-k$ is odd};\\
    \frac{1}{k!}\left(1 +\frac{2 k(k-1)}{(n-k+1)(n+k)}\right)
    &\text{ if $n-k$ is even}.
\end{cases}
\label{EqP1k}
\end{equation}
This formula was established by R. Du and R. Stanley
\cite[page 46]{StanleyTalkPP2010},
see also \cite[Equation (2)]{BernardiEtAlSeparation}.
For other compositions $I$, we are able to get information
of the denominator of $P^I_{n,0}$, seen as a function of $n$.
\begin{proposition}
 Fix a composition $I$ of size $m$ and length $k$. Then, there exist
 polynomials $Q^e_I(n)$ and $Q^o_I(n)$ such that, for  $n \ge m$,
 \[P^I_{n,0} = \begin{cases}
     \frac{Q^e_I(n)}{(n-m+1)(n-m+2) \dots (n-k) (n-k+1) (n+k) (n+k+1) \dots (n+m)} &\text{ if $n-k$ is even}\\
     \frac{Q^o_I(n)}{(n-m+1)(n-m+2) \dots (n-k) (n+k+1) \dots (n+m)  } &\text{ if $n-k$ is odd}.
 \end{cases}\]
\end{proposition}
\begin{proof}
    We proceed by induction over $m-k$.

    If $m=k$, that is if $I=(1^k)$, the result follows from the explicit expression \eqref{EqP1k}.

    For $m>k$, the result follows from \cref{EqIndA0}, the induction hypothesis
    and the fact that 
    $n(n+1)P^I_{n+1,1}$ coincide with a polynomial in $n$
    for even (respectively odd) values of $n$.
\end{proof}
Some examples computed by Du and Stanley \cite[page 52]{StanleyTalkPP2010}  
indicate that the above fractions are not in their reduced form
and that there are some nice cancellations that we did not manage to explain 
with our induction formula.\bigskip

In another direction,
after rescaling by defining
\begin{equation}
    \widetilde{P^I_{n,a}} = \frac{P^I_{n,a}}{\prod_{h=1}^k (i_h-1)!},
    \label{EqPtilde}
\end{equation}
separation probabilities exhibit the same kind of dependency that
we have observed in Theorem \ref{ThmDependence}.
As for partitions, we denote $m_i(I)$ the number of parts $i$ in $I$.
\begin{proposition}
    Fix $n$ and $a \ge 1$.
    The quantity $\widetilde{P^I_{n,a}}$ depends
    on $I$ only through its size $m$, its length $k$ and its small multiplicities 
    $m_1(I)$, \dots, $m_{a-1}(I)$.
    \label{PropDepPTilde}
\end{proposition}
\begin{proof}[Proof (by induction on $a$)]
    The base case $a = 1$ follows from the explicit formula
    \[ \widetilde{P^I_{n,a}} = \frac{1}{m!} \left( 1 + (-1)^{n-k+1} \frac{m(m-1)}{n(n-1)} \right),\]
    obtained from \cref{eq:PI,EqPn1,EqPtilde}.

    Suppose that the statement holds for some $a \ge 1$;
    that is, the rescaled separation probability $\widetilde{P^J_{n,a}}$
    is a function $F_{n,a}(m',k',m_1(J),\dots,m_{a-1}(J))$,
    where $m'$ and $k'$ are respectively the size and length of $J$. 
    Fix some composition $I = (i_1, \dots, i_k)$ with size $m$ and length $k$.
    Then, Theorem \ref{ThmRecBonaPb}, rewritten for $\widetilde{P^I_{n,a}}$, yields:
\[ n(n+1)\widetilde{P^I_{n+1,a+1}}
=(n-m+1)(n+m)\widetilde{P^I_{n,a}}
+ \sum_{1 \le h \le k \atop i_h \ne 1} i_h \, \widetilde{P^{I-\delta_h}_{n,a}}. \]
The first summand depends on $I$ only through $m$, $k$, $m_1(I)$, \dots, $m_{a-1}(I)$ by induction hypothesis,
so we focus on the last sum.
We split it according to the value of $i_h$: namely,
set
\[ \varSigma_b = \sum_{1 \le h \le k \atop i_h =b} i_h \, \widetilde{P^{I-\delta_h}_{n,a}}
= b \, m_b(I) \, \widetilde{P^{J_b}_{n,a}}. \]
The last equality comes from the fact that, for a fixed $b$, 
all compositions $J:=I-\delta_h$ 
are reordering of some composition $J_b$. 
Using the induction hypothesis, we get that
\[\widetilde{P^{J_b}_{n,a}} =\begin{cases}
    F_{n,a} \big(m-1,k,m_1(I), \dots, m_{b-1}(I)+1, m_b(I)-1, \dots,m_{a-1}(I)\big)
    &\text{for }2 \le b \le a-1 \\
    F_{n,a} \big(m-1,k,m_1(I), \dots, m_{a-2}(I), m_{a-1}(I)-1\big) &\text{for }b=a \\
    F_{n,a} \big(m-1,k,m_1(I), \dots,m_{a-1}(I)\big) &\text{for }b>a.
\end{cases}\]
Therefore, for a fixed $b \le a$,
the quantity $\varSigma_b$ depends on $I$ only through $m$, $k$, $m_1(I)$, \dots, $m_a(I)$.
Moreover, we have
\begin{multline*}
    \sum_{b>a} \varSigma_b = \left( \sum_{b>a} b \, m_b(I) \right) \,
    F_{n,a} \big(m-1,k,m_1(I), \dots,m_{a-1}(I)\big) \\
= \left( m - \sum_{b \le a} b \, m_b(I) \right) \,
F_{n,a} \big(m-1,k,m_1(I), \dots,m_{a-1}(I)\big),
\end{multline*}
so that this quantity also depends only on $I$ only through $m$, $k$, $m_1(I)$, \dots, $m_a(I)$.
Hence the proposition holds for $a+1$.
\end{proof}

\begin{remark}
    {\em Weak} separation probability (in opposition to {\em strong}
    separation probability as considered here) fulfils a much 
    stronger symmetry property; see \cite[Equation $(1)$]{BernardiEtAlSeparation}.
\end{remark}

%
%
%
%

\section{Products of cycles of arbitrary length}\label{sec:cycles}

\subsection{Induction relations}\label{sec:realinduction}

In this section, we present lemmas which will be needed
in our inductive approach for product of cycles of arbitrary length.
Our approach in this section will be similar to that given in the
proof of Lemma \ref{LemAppliedInduction}, but be more involved because it will
apply more generally.

Let $\lambda$ and $\mu$ be partitions of $n$.
Define $A(m,\lambda, \mu)$ to be the tuples $(\sigma, \alpha, \beta)$ in
$\Sym{n}$ such that their product satisfies $\sigma \cdot \alpha \cdot \beta = e$, where
$e$ is the identity permutation, the
permutations $\alpha, \beta$ have cycle type $\lambda$ and $\mu$,
respectively, and $\sigma$ has $m$ cycles.
Define also $A_i(m, \lambda, \mu)$
to be the tuples in $A(m, \lambda, \mu)$ where $\alpha$ and $\beta$ have
precisely $i$ common fixed points.   In general, we use lower case $a$ to mean the
cardinality of the corresponding set with capital $A$ (i.e. $a_i(m, \lambda, \mu) =
\# A_i(m, \lambda, \mu)$ and $a(m, \lambda, \mu) = \# A(m, \lambda, \mu)$,
\emph{etc}. ).

It will be convenient in our induction to consider permutations of a set different from $[n]$.
Denote $\Sym{S}$ the set of bijections from $S$ to $S$.
The notation above is naturally extended to the case of permutations
with ground set $S$ by adding a superscript;
that is, $A^S(m,\lambda, \mu)$ (respectively $A^S_i(m, \lambda, \mu)$) is the set
of triples $(\sigma,\alpha, \beta)$ of permutations of the ground set $S$
with the above properties.

Before we begin our induction results, we make two simple observations regarding the
sets $A_i(m, \lambda, \mu)$.   For a partition $\lambda$ with at least $r$ 1s, we will denote the
partition obtained from $\lambda$ by removing $r$ 1s by $\lambda_{[r]}$ and,
recall from Section \ref{SectInduction}, the
partition $\lambda^{\downi{j}}$ is obtained from $\lambda$ by
replacing a part of size $j$ with a part of size $j-1$ if $m_j(\lambda) > 0$.  Thus, the two
notations agree at $\lambda^{\downi{1}}$ and $\lambda_{[1]}$.


\begin{lemma}
	Suppose that $n \geq 1$ and that $0 \leq r \leq \mathrm{min}\{ m_1(\lambda),
	m_1(\mu) \}$.  Then
	\begin{align*}
		A^{[n]}_r(m, \lambda, \mu) &= \bigcup_{S \subseteq [n] \atop
		|S| = r} A^{[n] \setminus S}_0(m-r, \lambda_{[r]}, \mu_{[r]})\;\;\;\;
		\mathrm{ and }\\ a_r(m, \lambda, \mu) &= {n \choose r}
		a_0(m-r, \lambda_{[r]}, \mu_{[r]}).
	\end{align*}
	\label{lem:simple}
\end{lemma}
\begin{proof}
	The claim in the lemma is merely that the $r$ common fixed points of $\alpha$ and
	$\beta$ will also be fixed points of $\sigma$.  We may therefore remove these
	common fixed points from the triple and consider the triple $(\sigma,
	\alpha, \beta)$ in $A^{[n] \setminus S}(m-r, \lambda_{[r]}, \mu_{[r]})$, where $S$ is the set
	of common fixed points.  The second equation follows from the first.
\end{proof}
\begin{lemma}\label{lem:triv}
	For $n \geq 1$, we have
	\begin{equation*}
		A(m, \lambda, \mu) = \bigcup^{\mathrm{min}\{m_1(\lambda),
		m_1(\mu)\}}_{r=0} A_r(m, \lambda, \mu)
	\end{equation*}
\end{lemma}
\begin{proof}
	Follows from the definitions of the relevant objects.
\end{proof}
We therefore see from Lemmas \ref{lem:simple} and \ref{lem:triv} that evaluating
$a_0(m-r, \lambda_{[r]}, \mu_{[r]})$ will help us obtain our formulae.

Our main induction theorem will determine $A_0(m, \lambda, \mu)$.  We will
find it useful to use rooted versions of the factorisations introduced.  Let
$A_i(m, {\lambda}, \dot{\mu})$ be the set of all triples $(\sigma, \alpha, \beta) \in
A_i(m, {\lambda}, \mu)$
where we have distinguished one fixed point of $\beta$.  Similarly, for $j \geq 2$
with $m_j(\lambda)>0$, let
$A_i(m, \hat{\lambda}^\downi{j}, \mu)$ be the set of triples $(\sigma, \alpha, \beta)
\in A(m, \lambda^\downi{j}, \mu)$, with one member of an $(j-1)$-cycle distinguished,
and with precisely $i$ fixed points in common.  When $j=2$, this means
that we distinguish a fixed point of $\alpha$.
If it is also a fixed point of $\beta$, we do not take it into account
when counting the common fixed points of $\alpha$ and $\beta$
(thus, we should say that $\alpha$
and $\beta$ have precisely $i$ non-rooted fixed points in common).  The following
decomposition holds.
\begin{lemma}
	Suppose that $\mu, \lambda \vdash n$.  If $m_1(\mu) \neq 0$, we have
	\begin{equation*}
		A^{[n]}_0(m, {\lambda}, \dot{\mu}) = \bigcup_{s \in [n]} \ \bigcup_{j
		\geq 2 \atop m_j(\mu)>0} \ A^{[n] \setminus s}_0(m, \hat{\lambda}^\downi{j}, {\mu}_{[1]}).
	\end{equation*}
	\label{thm:onestage}
\end{lemma}

\begin{proof}
	Define a function $\psi$ mapping the left hand side to the right hand side by
	the following rule.  Suppose that $s$ is the rooted fixed point of
	${\beta}$ and suppose that $s$ is contained in a $j$-cycle of $\alpha$, for some
    $j \geq 2$ with $m_j(\lambda)>0$.  Since $\alpha$ and $\beta$ have no common
	fixed points, we know that $\alpha(s) \neq s$.  
	Define $\hata$ as $\left(s\; \alpha(s) \right) \alpha$ and $\hatb$ as $\beta$ with
	the fixed point $s$ removed.  Also, let $\hats = \left(s\;
	\sigma(s)\right) \sigma$.
	Finally, let $\psi(\sigma, \alpha, \beta) =
	(\hats, \hata, \hatb)$, where the root of $\hata$ is chosen to be 
	$\alpha(s)$.   We show that $\psi$ is
	well-defined and bijective.  	

	To show that $\psi$ is well defined we must show the following. As defined, $\hats, \hata$ and $\hatb$ are
	permutations with ground set $[n]$;  note, however, that $s$ is clearly a 
	fixed point of all of them, and thus we can consider those permutations in
	$\Sym{[n] \setminus \{s\}}$.  Therefore, we must show that
	$\hats \cdot \hata \cdot \hatb = e$ in $\Sym{[n] \setminus \{s\}}$, $\hata$ and $\hatb$ have
	cycle types $\lambda^\downi{j}$ and $\mu_{[1]}$, respectively, and $\hats$
	has $m$ cycles.   We show these in turn.

	The following computation show that the product of $\hats, \hata$ and
	$\hatb$ is the identity.
	\begin{align*}
		\hats \cdot \hata \cdot \hatb &= (s\; \sigma(s)) \cdot \sigma
		\cdot  (s\; \alpha(s)) \cdot \alpha \cdot \beta\\
		&= (s\; \sigma(s) ) \cdot (\sigma(s)\;\; \sigma \circ \alpha(s))
		\cdot \sigma \cdot \alpha \cdot \beta\\
		&= \sigma \cdot \alpha \cdot \beta\\
		&=  e,
	\end{align*}
	where the third equality follows from $s  = \sigma \circ \alpha \circ
	\beta(s) = \sigma \circ \alpha (s)$.  Thus, in $\Sym{[n]}$ we see that the
	product of $\hats, \hata$ and $\hatb$ is $e$, hence also true in
	$\Sym{[n] \setminus \{s\}}$ since they each fix $s$.
	
	Clearly, $\hatb$ is of the claimed cycle type, and we can additionally see
	that $(s\; \alpha(s))$ removes $s$ from the cycle in $\alpha$ that
	contains it and leaves
	the rest of the cycle unchanged, showing that $\hata \in
	\lambda^\downi{j}$ for some $j \geq 2$ and $m_j(\lambda) > 0$.  To
	show that $\hats$ has $m$ cycles, note that $(s\; \sigma(s))$ removes $s$ from
	the cycle in $\bar{\sigma}$ that contains it, leaving the rest of the
	cycle unchanged.  We should
	take care to ensure that $s$ is not the only member of its cycle in
	$\sigma$;  however,
	since $\alpha(s) \neq s$ and $\beta(s) = s$, we see that $\sigma(s) \neq s$,
	for otherwise the product of $\sigma, \alpha$ and $\beta$ would not be the
	identity.  Thus, we see that $\hats$ has $m$ cycles.  This completes the proof.
\end{proof}

We wish to apply the operations in Lemma \ref{thm:onestage} twice in succession.  To
do this we introduce the following set.  For $i, j \geq 2$
with $m_i(\lambda),m_j(\mu)>0$, let $A_0(m,\hat{\lambda}^\downi{i},
\hat{\mu}^\downi{j})$ be the set of triples $(\sigma, \alpha, \beta) \in
A(m, \lambda^\downi{i}, \mu^\downi{j})$ such that a member of an $i-1$ and $j-1$-cycle
are roots of $\alpha$ and $\beta$, respectively, and $\alpha$ and $\beta$ have no non-rooted fixed points in common (thus,
if $i=2$,  we are to root a fixed point of $\alpha$ and a rooted fixed point of $\alpha$ can occur as a fixed point in $\beta$).  We have the following lemma.
\begin{lemma}
	Let $\lambda, \mu \vdash n$ and both $m_1(\lambda), m_1(\mu) \neq 0$.  Then,
	\begin{equation*}
		A^{[n]}_0(m, \dot{\lambda}, \dot{\mu}) = 
		\bigcup_{s,t \in [n] \atop s \neq t} \bigcup_{ i,j \geq 2 \atop
		m_i(\lambda), m_j(\mu) > 0}
		A^{[n] \setminus \{s,t\}}_0  (m, \hat{\lambda_{[1]}}^\downi{i}, \hat{\mu_{[1]}}^\downi{j}).
	\end{equation*}
	\label{thm:twostage}
\end{lemma}
\begin{proof}
	Define a function $\psi$ mapping the left hand side to the right hand side by
	the following rule.  Suppose that $s$ and $t$ are the rooted fixed points of
	${\alpha}$ and ${\beta}$, respectively.
	Define $\hata = (t\; \alpha(t)) \alpha$ and $\hatb = (s\; \beta(s)) \beta$.
	Also let $p = \hata \circ \beta(s)$ and define 
	\begin{equation}
		\hats = \sigma \cdot (t\; \alpha(t)) \cdot (s\; p).
		\label{eq:hats}
	\end{equation}
	Finally, let $\psi(\sigma, \alpha, \beta) =
	(\hats, \hata, \hatb)$.  We choose as the root of $\hata$ to be
	$\alpha(t)$ and for $\hatb$ we choose $\beta(s)$.  
	
	To show that $\psi$ is well defined we must show the following. As defined, $\hats, \hata$ and $\hatb$ are
	permutations with ground set $[n]$;  we must, however, show that $s,t$ are
	fixed points of all of them, and thus we can consider those permutations in
	$\Sym{[n] \setminus \{s,t\}}$.  We must further show that
	$\hats \hata \hatb = e$ in $\Sym{[n] \setminus \{s,t\}}$, $\hata$ and $\hatb$ have
	cycle types $\lambda^{\downi{i}}$ and $\mu^{\downi{j}}$, respectively, and $\hats$
	has $m$ cycles.   The proof, however, largely follows the proof of Lemma
	\ref{thm:onestage}, so we omit the details.
\end{proof}
Note that it is straightforward to use Lemma \ref{thm:onestage} when
$m_2(\lambda) = 0$.  In that case, it immediately follows from Lemma
\ref{thm:onestage} that if $\lambda, \mu \vdash n$ with $m_1(\mu) \neq 0$, we have
\begin{equation}\label{eq:enumonestage}
	m_1(\mu) a_0(m, \lambda, \mu) = n \sum_{j \geq 3 \atop m_j(\lambda) >0} a_0(m, \lambda^\downi{j},
	\mu^\downi{1}) \cdot (m_{j-1}(\lambda) + 1)\cdot (j-1).
\end{equation}
If $m_2(\lambda) \neq 0$, the complication is that a part of size 1 is created in 
$\lambda^\downi{2}$.  This corresponds to a fixed point being created in a relevant
permutation $\alpha$ and this fixed point may occur in $\beta$.

Similarly, we can use easily Lemma \ref{thm:twostage} when $m_2(\lambda)=m_2(\mu)=0$.
If, furthermore, one has $m_1(\mu),m_1(\lambda) \neq 0$, then
\begin{multline}\label{EqEnumTwoStage}       
    m_1(\lambda) m_1(\mu) a_0(m, \lambda, \mu) = (n)_2
    \sum_{i,j \geq 3 \atop m_i(\lambda),m_j(\mu)>0} a_0(m, \lambda_{[1]}^\downi{i},      
    \mu_{[1]}^\downi{j}) \cdot (m_{i-1}(\lambda) + 1)\\
    \cdot (i-1) \cdot 
(m_{j-1}(\mu) + 1)\cdot (j-1).                          
\end{multline}
Note that this equation can be used only when $m_2(\lambda)=m_2(\mu)=0$.
We shall be mindful
of this technicality, but will be able to overcome it easily.\bigskip

The reader has certainly noticed that the arguments used in this Section
and in the proof of \cref{LemAppliedInduction} are similar.
In fact, the following variant of Lemma \ref{thm:onestage} and its Corollary
generalizes \cref{LemAppliedInduction}.   We will, however, primarily be using
Lemmas \ref{thm:onestage} and \ref{thm:twostage} in the coming sections.
\begin{lemma}
	Suppose that $\mu, \lambda \vdash n$, and $m_1(\mu) \neq 0$.  Then
	\begin{equation*}
		A^{[n]}(m, \lambda, \dot{\mu}) = \bigcup_{s \in [n]} \bigcup_{j
		\geq 2 \atop m_j(\lambda) >0} A^{[n] \setminus s}(m, \hat{\lambda}^\downi{j}, \mu_{[1]})
		\cup \bigcup_{s \in [n]} A^{[n] \setminus s}(m-1, \lambda_{[1]}, \mu_{[1]}),
	\end{equation*}
    where the last union should be omitted if $m_1(\lambda)=0$ or $m=1$.
	\label{thm:onestageno}
\end{lemma}
\begin{proof}
	The proof is similar to that of Lemma \ref{thm:onestage}, except that a
	triple $(\sigma, \alpha, \beta)$ can all have the same fixed point.  However,
	that is taken care of in the second union in the lemma.
\end{proof}

Counting the sets on both sides of Lemma \ref{thm:onestageno} gives the following.

\begin{corollary}
	Suppose that $\lambda, \mu \vdash n$, and that $m_1(\mu) \neq 0$.  Then,
	\begin{multline*}
		m_1(\mu)\e a(m, \lambda, \mu) = n \sum_{j \geq 2 \atop m_j(\lambda) >0} a(m,
		\lambda^\downi{j}, \mu^\downi{1}) \cdot (m_{j-1}(\lambda) + 1) \cdot
		(j-1)+
		n\e a(m-1, \lambda^\downi{1}, \mu^\downi{1}),
	\end{multline*}\label{thm:corschaeff}
    where the last term should be omitted if $m_1(\lambda)=0$ or $m=1$.
\end{corollary}
\cref{LemAppliedInduction} is a special case of this Corollary
\ref{thm:corschaeff} (with $m=1$ and $\mu=(n-a,1^{a+1})$).
We do not use \cref{thm:corschaeff} in its full generality in this paper.
However, we hope that it can be used in the future to give generalisations of
Lemma \ref{LemAppliedInduction}.

\subsection{A formula for multiplying cycles}
We use the ideas developed in Section \ref{sec:realinduction}, to obtain enumerative results
for multiplying cycles;  specifically, we look for formulae for 
$a\left(m, (i+t, 1^{j-t}), (i, 1^j))\right)$ for $i \geq 1$ and $0 \leq t \leq j$.  
To simplify notation, we will omit the parenthesis and commas in the notation
of the partitions above: that is we shorten the notation to
$a(m, i+t \e 1^{j-t}, i \e 1^j )$.

First, note that, because of the sign function on $\Sym{n}$,
$a(m, i \e 1^j, i+t \e 1^{j-t})=0$ unless $m+i+j+t$ is even.

Second, the case $j=t=0$ is known;
see {\em e.g.} \cite[Corollary 3.4]{StanleyEnumerativePermutations}.
For any $r \geq 1$ and $m$, we have
\begin{equation}
	\frac{a(m, r, r)}{(r-1)!} = \begin{cases}
		\frac{c(r+1,m)}{{r+1 \choose 2}} &\text{ if $r-m$ is even}\\
        0 &\text{otherwise,}
    \end{cases}
    \label{EqDistribProductLongCycles}
\end{equation}
where $c(r+1,m)$ is a signless Stirling number of the first kind;
that is, it is the number of permutations on $[r+1]$ with $m$ cycles.

Thus, we shall attempt to find a formula for $a(m, i+t\e
1^{j-t}, i\e 1^j)$ in terms of $a(s, r, r)$ for $r \leq i + j$ and $s \leq m$.
We first remark that we can easily determine a formula for $i=1$ and $0 \leq t
\leq j$, since
\begin{equation*}
    a(m, 1+t\e 1^{j-t}, 1^{j+1}) = \frac{(j+1)!}{(1+t)(j-t)!} \delta_{m,j-t+1},
\end{equation*}
	where $\delta$ is the usual Kronecker $\delta$ function.
We may therefore assume that $i \geq 2$.  We have the following results when $j-t =
0$.
\begin{lemma}
	Suppose that $i \geq 2$ and $j \geq 0$.  Then
	\begin{equation*}
		a(m, i+j, i\e 1^j) = \frac{(i + j)_j (i + j -1)_j}{j!} a(m, i, i).
	\end{equation*}
	\label{thm:onestagecyc}
\end{lemma}
\begin{proof}
	We prove the claim by induction.  The claim is trivial when $j=0$,
	establishing the base case.  
    Notice that a permutation of type $(i+j)$ has no fixed points, so
	$a(m, i+j, i\e 1^j) = a_0(m, i+j, i\e 1^j)$.  Now assuming that $j > 0$,
	note that $i+j
	\geq 3$, and we can apply Lemma \ref{thm:onestage} as we did in
	\eqref{eq:enumonestage}.  We get that
	\begin{equation*}
		a(m, i+j, i\e 1^{j}) = \frac{(i+j) \cdot (i+j-1)}{j} a(m, i+j - 1,
		i\e 1^{j-1}).
	\end{equation*}
	Applying induction gives the desired formula.
\end{proof}
Notice that if $i-j + t < 0$, or equivalently $j > i+t$, 
that $a_0(m, i + t\e 1^{j-t}, i\e 1^j) = 0$.  Keeping that in mind, we have the
following lemma.
\begin{lemma}
	Suppose that $i \geq 2$ and $0 \leq t \leq j \leq i+ t$.  If $i-j+t \geq
	1$, we have
	\begin{equation*}
		a_0(m, i+t\e 1^{j-t}, i\e 1^j) = \frac{(i+j)!}{(i-j+t)!}
		{i-1 \choose j-t} {i+t-1 \choose j} a(m, i-j+t, i-j+t),
	\end{equation*}
	and if $i-j+t = 0$, we have
	\begin{equation*}
		a_0(m, i+t\e 1^{j-t}, i\e 1^j) = \frac{(i+j)!}{i \cdot (i+t)}
		\delta_{m,2}.
	\end{equation*}
	\label{thm:mainsummand}
\end{lemma}
\begin{proof}
    Substituting $\lambda=(r,1^{n-r})$ and $\mu=(s,1^{n-s})$ in \eqref{EqEnumTwoStage}
    yields that
    \begin{equation}
        a_0(m,r\e 1^{n-r},s\e 1^{n-s}) =\frac{n(n-1) \cdot (r-1)(s-1)}{(n-r)(n-s)}
        a_0(m,r-1 \e1^{n-1-r},s-1\e 1^{n-1-s}).
        \label{EqRecAmHooks}
    \end{equation}
    But this equation is only valid if $s,r\ge 3$.
    The idea is to iterate it as long as possible, and deal with the two special
    cases at the end.

    If $i - j + t \geq 2$, we can iterate \eqref{EqRecAmHooks} without any
	complication.  If $a_0(m, i\e 1^j, i+t\e
	1^{j-t})$ is non-zero, we have
	\begin{multline}
		a_0(m, i+t\e 1^{j-t},i\e 1^j) =
        \frac{(i+j) (i+j-1) (i-1) (i+t-1)}{j(j-t)} a_0(m, i+t-1\e
		1^{j-t-1},i-1 \e 1^{j-1})\\
		= \frac{(i+j) (i+j-1) (i-1) (i+t-1)}{j(j-t)}
		\cdots \frac{(i - j + t
		+2) (i-j+t + 1)(i-j+t) (i-j+2t)}{(t+1)1} \\
       \cdot {a_0(m, i-j+2t,i-j+t\e 1^t)} \\
		= \frac{(i+j)_{2(j-t)} (i-1)_{j-t} (i+t-1)_{j-t}}{(j)_{j-t}
		(j-t)_{j-t}}{a_0(m, i-j+2t,i-j+t\e 1^t)}.
        \label{eq:repeat}
	\end{multline}
	Therefore, we have
	\begin{multline*}
		a_0(m, i + t\e 1^{j-t},i\e 1^j) =
		\frac{(i+j)_{2(j-t)} (i-1)_{j-t} (i+t-1)_{j-t}}{(j)_{j-t}
		(j-t)_{j-t}} a(m, i-j+2t,i-j + t\e 1^t)\\
		= \frac{(i+j)!}{(i-j+2t)!} {i-1 \choose j-t}
		\frac{(i+t-1)!}{(i-j+2t-1)! \frac{j!}{t!}} \frac{(i-j+2t)_{t}
		(i-j+2t-1)_{t}}{t!}a(m, i-j+t, i-j+t)\\
		= \frac{(i+j)!}{(i-j+t)!} {i-1 \choose j-t}{ i+t-1 \choose j}a(m, i-j+t, i-j+t).
	\end{multline*}
    If $i-j+t \leq 1$, we cannot use equation \eqref{EqEnumTwoStage} iteratively
    as long as needed.  We can, however, explicitly count the final stages of
    the iteration given by $a_0(m, r\e 1^{n-r}, 2\e 1^{s})$.
    In fact, only the cases $r=n-2$ and $r=n-1$ will be needed below.

    If $r=n-2$, we
	have by explicitly counting 
	\begin{equation*}
		a_0(m, n-2\e 1^{2},2\e 1^{n-2}) = \frac{n!}{2(n-2)}
		\delta_{m,2}
	\end{equation*}
	and if $r = n-1$, we have
	\begin{equation*}
        a_0(m, n-1 \e 1,2\e 1^{n-2}) = n!\ \delta_{m,1} =n! \; a(m, 1,1).
	\end{equation*}
	Thus, when $i-j+t = 1$, we have by repeating the computation in \eqref{eq:repeat}
	\begin{align*}
		a_0(m, i + t\e 1^{j-t},i\e 1^j) &=
		\frac{(i+j)_{2(i-2)} (i-1)_{i-2} (i+t-1)_{i-2}}{(j)_{i-2}
		(j-t)_{i-2}}\\
		&\;\;\;\;\;\;\;\;\;\;\;\;\;\;\;\;\;\;\;\;\;\;\;\;\;\;\;\; \cdot
		a_0(m,t+2\e 1, 2\e 1^{j-i+2}, )\\
		&= \frac{(i+j)_{2(i-2)} (i-1)_{i-2} (i+t-1)_{i-2}}{(j)_{i-2}
		(j-t)_{i-2}} (j-i+4)!\  a(m,1,1)\\
		&= (i+j)!\  a(m,1,1),
	\end{align*}
	which is the claimed formula when $i-j+t = 1$.  If $i - j+t= 0$, repeating
	the computation in \eqref{eq:repeat} gives
	\begin{align*}
		a_0(m, i + t\e 1^{j-t},i\e 1^j) &=
		\frac{(i+j)_{2(i-2)} (i-1)_{i-2} (i+t-1)_{i-2}}{(j)_{i-2}
		(j-t)_{i-2}}\\
		& \;\;\;\;\;\;\;\;\;\;\;\;\;\;\;\;\;\;\;\;\;\;\;\;\;\; 
        \cdot a_0(m, t+2\e 1^2, 2\e 1^{j-i+2})\\
		&= \frac{(i+j)_{2(i-2)} (i-1)_{i-2} (i+t-1)_{i-2}}{(j)_{i-2}
		(j-t)_{i-2}} \frac{(j-i+4)!}{2(j-i+2)} \delta_{m,2}\\
		&= \frac{(i+j)!}{i \cdot (i+t)} \delta_{m,2}.
	\end{align*}
	This completes the proof.	
\end{proof}
This lemma gives an explicit formula
for $a(m,i \e 1^j, i+t\e 1^{j-t})$.
\begin{theorem}\label{thm:maincount}
	Suppose that $i \geq 2$ and $0 \leq t \leq j$,
	and further suppose that $i+j+t+m$ is even. Then,
	\begin{align*}
		a(&m, i + t\e 1^{j-t}, i\e 1^j) = \\
		&\sum_{s=\mathrm{max}\{0, j-i-t
		+1\}}^{j-t} \frac{(i+j)!}{s! (i-j+t + s)} {i-1 \choose j-s-t} {i+t
		-1 \choose j-s} 
		\frac{c(i-j+t+s+1,m-s)}{{i-j+t+s+1 \choose 2}}\\
		&+ \Lambda(m,i,j,t),
	\end{align*}
	where
	\begin{equation*}
		\Lambda(m,i,j,t) = 
		\begin{cases}
			0 & \mathrm{ if }\; i - j + t \geq 1;\\
			\frac{(i +j)!}{(j-i-t)! i \cdot (i+t)}
			\delta_{m-j+i+t,2} & \mathrm{
			if }\; i - j + t \le  0.
		\end{cases}
	\end{equation*}
\end{theorem}
\begin{proof}
	We have from Lemmas \ref{lem:simple} and \ref{lem:triv}, that
	\begin{equation}\label{eq:substheorem}
		a(m, i+t\e 1^{j-t}, i\e 1^j) = 
		\sum_{s} {i+j \choose s} a_0(m-s, i+t\e 1^{j-t-s}, i\e 1^{j-s}).
	\end{equation}
	Notice that if $i - j + t \leq 0$, then $a_0(m-s, i+t\e 1^{j-t-s}, i\e 1^{j-s}) = 0$ unless $s \geq j-i -t$.  Hence the range for $s$ in the theorem.
	Applying Lemma \ref{thm:mainsummand} to the summand in
    \eqref{eq:substheorem} and using formula \eqref{EqDistribProductLongCycles}
    gives the desired formula.
\end{proof}
We mention a particularly interesting case of multiplying two $i$-cycles in
$\Sym{n}$.
\begin{corollary}\label{mainformula}
	Suppose that $i \geq 2$, $j \geq 0$ and $i+j+m$ is even.  Then,
	\begin{align*}
		a(&m, i\e 1^j, i\e 1^{j}) = \\
		&\sum_{s=\mathrm{max}\{0, j-i +1\}}^{j} \frac{(i+j)!}{s! (i-j
		+ s)}
		{i-1 \choose j-s}^2 \frac{c(i-j+s+1, m-s)}{ {i-j+s+1 \choose 2}
	}+ \Lambda(m,i,j),
	\end{align*}
	where
	\begin{equation*}
		\Lambda(m,i,j) = 
		\begin{cases}
			0 & \mathrm{ if }\; i - j > 0;\\
			\frac{(i + j)!}{(j-i)! i^2} \delta_{m-j+i,2} & \mathrm{
			if }\; i - j \leq 0.
		\end{cases}
	\end{equation*}
\end{corollary}

\subsection{Separation probabilities when neither cycle is full}\label{sec:nonfullsep}

\newcommand{\eye}{i}
In this section, we use the standard notion of separable permutations, as opposed to
strong separation used in most of Section \ref{SectSeparability}.  In this
section, we use Lemmas \ref{lem:simple} and \ref{thm:twostage} to determine the separation probability for
an arbitrary composition $I$ when multiplying two $(n-1)$-cycles in
$\Sym{n}$.  To that end, let $\lambda, \mu \vdash n$ and $I = (\eye_1,
\eye_2, \dots \eye_k)$ be a composition of $m \leq n$.
We use the notation of \cite{BernardiEtAlSeparation}: 
$\sigma^I_{\lambda,\mu}$ is
the probability that the product of two uniformly distributed permutations of
cycle types $\lambda$ and $\mu$, respectively, is $I$-separated.  When $\lambda
= \mu = (n)$, the following was proved in \cite{BernardiEtAlSeparation}:
\begin{equation}
	\sigma^I_{(n), (n)} = \frac{(n-m)! \prod (\eye_j)!}{(n-1)!^2} D^{m,k}_n,
	\label{eq:stanprob}
\end{equation}
where 
\begin{equation*}
	D^{m,k}_n=\frac{(n-1)!}{(n+k) {n+m \choose m-k}} \left( (-1)^{n-m} {n-1 \choose k-2} + \sum_{r=0}^{m-k} {n+m \choose n+k+r} {n+r +1 \choose m} \right).
\end{equation*}
We are chiefly interested in the case $\lambda=\mu=(n)$ or $(n-1,1)$.
Let $S^I_\lambda$ be the number of permutations that are $I$-separated over
all products of two permutations of cycle type $\lambda$.  Thus, 
$$S^I_{(n)}
= (n-1)!^2 \sigma^I_{(n), (n)} = (n-m)! \prod_{j} (\eye_j)! D^{m,k}_n.$$ 

Let $A^{[n]}(\lambda) = \cup_{m} A^{[n]}(m, \lambda, \lambda)$;  that is, 
$A^{[n]}(\lambda)$ it is 
set of all triples $(\sigma, \alpha, \beta)$ such that $\sigma \alpha \beta = e$
in $\Sym{n}$ and $\alpha$ and $\beta$ have cycle type $\lambda$.  Our  
goal is to determine the number of such triples in $A^{[n]}\left( (n-1,1)
\right)$ where $\sigma$ is $I$-separated.
Using the decompositions in Lemmas \ref{lem:simple} and
\ref{thm:twostage} we see that the elements in $A^{[n]}( (n-1,1))$ come in two
forms;  either the triple $(\sigma, \alpha, \beta)$ has a common fixed point or
it does not.  We take these two cases separately, with the second case being a
lot more
involved.  In speaking of separation below, we take $J = (J_1, \ldots, J_k)$
to be the set partition of $\{1,2, \ldots, m\}$, where $J_p = \{\eye_{p-1} + 1,
\dots, \eye_p\}$.  However, we remind ourselves that any set partition $J$ whose
block sizes realise the composition $I$ will have $S^J_{\lambda} = S^I_{\lambda}$.  We use this fact repeatedly below.

If $(\sigma, \alpha, \beta)$ has a common fixed point $c$, then by Lemma
\ref{lem:simple} these triples are in
bijection with triples $(\hat{\sigma}, \hat{\alpha}, \hat{\beta})$ in $A^{[n] \setminus c}\left( (n-1) \right)$.  Thus, if $c \in [m]$, $\sigma$ is $J$-separated if
and only if $\hat{\sigma}$ is $J \setminus c$-separated, and if $c \notin [m]$ then
$\sigma$ is $J$-separated if and only if $\hat{\sigma}$ is $J$-separated.  We
have a total of $\sum_{j=1}^k \eye_j S^{I \downarrow (\eye_j)}_{(n-1)}$ in the former
case, and 
\begin{equation}
	(n-m) S^I_{(n-1)}
	\label{eq:partexp1}
\end{equation}
in the latter case.  The former simplifies as
\begin{align}
	\sum_{j=1}^k \eye_j S^{I \downalphai{j}}_{(n-1)} &= \sum_{j=1 \atop
	\eye_j \neq 1}^k \eye_j S^{I \downalphai{j}}_{(n-1)} + \sum_{j=1
	\atop \eye_j = 1}^k \eye_j S^{I \downalphai{j}}_{(n-1)}\notag \\
	&= \sum_{j=1 \atop \eye_j \neq 1}^k \eye_j (n-m)! (\eye_j - 1)! \prod_{p=1
	\atop p \neq j}^k (\eye_p)! D^{m-1,k}_n
	 +\sum_{j=1
	\atop \eye_j = 1}^k \eye_j (n-m)! \prod_{p=1 \atop p \neq j}^k
	(\eye_p)! D^{m-1,k-1}_{n-1}\notag\\
	&= (n-m)! \prod_{j=1}^k (\eye_j)! \left( (k - m_1(I))
	D^{m-1,k}_{n-1} + m_1(I) D^{m-1,k-1}_{n-1}\right)\label{eq:partexp2}.
\end{align}

We see from
Lemma \ref{thm:twostage} that the set of triples $(\sigma, \alpha, \beta)$ that do not have a common fixed point is in bijective
correspondence with the triples in $\cup_{s,t \in [n] \atop s \neq t}
A^{[n] \setminus \{s,t\}}( (n-2))$,  and $(\sigma, \alpha, \beta)$ corresponds to
$(\hat{\sigma}, \hat{\alpha}, \hat{\beta})$ where $\sigma = (s\; u) (t\; v)
\hat{\sigma}$, for some $u,v \in [n] \setminus \{s,t\}$ not necessarily distinct.  Of course, the effect of $(s\; u)$ and $(t\; v)$ is
to bring $s$ into the cycle containing $u$ and $t$ into the
cycle containing $v$ in $\hat{\sigma}$.  Thus, enumerating the triples $(\sigma, \alpha,
\beta)$ with no common fixed point is equivalent to enumerating tuples $(s,t,u,v,
\hat{\sigma}, \hat{\alpha}, \hat{\beta})$, where $s,t \in [n]$ are distinct, $u,v \in
[n] \setminus \{s,t\}$ are not necessarily distinct, and $(\hat{\sigma}, \hat{\alpha}, \hat{\beta}) \in
A^{[n] \setminus \{s,t\}}( ( n-2))$. Since we are often not concerned with $\hat{\alpha}$ and 
$\hat{\beta}$ (just that they are part of the relevant triple), we will
generally omit writing them.  Furthermore, we are trying to determine in this
correspondence when $\sigma$ is $I$-separated, which will rely on a case
analysis on $s,t,u,v$ and on the separation probabilities of $\hat{\sigma}$.
We discuss this in the three cases given by the size of the
intersection of $\{s,t\}$ with $[m]$.

If $\{s,t\}$ is disjoint from $[m]$, then $\sigma$ is $J$-separated if and only
if $\hat{\sigma}$ is $J$-separated in $\Sym{[n] \setminus \{s,t\}}$.  The number of
quadruples $s,t,u,v$ for which $\{s,t\}$ are disjoint from $[m]$ are $(n-m)_2
(n-2)^2$.  Thus, the total number of tuples $(s,t,u,v, \hat{\sigma})$ in this
case is 
\begin{equation}
	(n-m)_2 (n-2)^2 S^{I}_{(n-2) }
	\label{eq:partexp0}
\end{equation}

Now suppose that exactly one of $s,t \in [m]$, and for now assume it is $s$ (the case for
$t$ is the same).  Since $s \in [m]$, then $s \in J_j$ for some $j$.  Recall,
the effect of $(s\; u)$ on $\hat{\sigma}$ is to bring $s$ into the cycle
containing $u$ in $\hat{\sigma}$.  If $u \notin [m]$ we see that $\sigma$
is $J$-separated if $\hat{\sigma}$ is $J(s,u)$-separated, where $J(s,u)$ is the
set partition with $s$ replaced by $u$.  The number of such
$\hat{\sigma}$ is $S^{I}_{(n-2) }$ and the number of quadruples $(s,t,u,v)$ when
$u \notin [m]$ is
$ \eye_j (n-m) (n-m-1)(n-2)$ (each factor gives the number of choices for
$s,t,u,v$, respectively), giving a total of $\eye_j (n-2) (n-m)_2
S^{I}_{(n-2)}$.  If $u \in [m]$, then $u$ must be in the same part as $s$,
and $\sigma$ is $J$-separated if and only if $\hat{\sigma}$ is
$J \setminus \{s\}$-separated.  Thus, the number of choices of $s,t,u,v$ when
$u \in [m]$ is
$\eye_j (n-m) (\eye_j - 1) (n-2)$, which gives a total contribution of 
$\eye_j(\eye_j - 1) (n-m)(n-2) S^{I \downalphai{j}}$ in this case.  Summing over all $j$ and
multiplying by two to account for the case $t \in [m]$ and $s \notin [m]$ we get
\begin{equation}
	2 (n-2) (n-m)_2 \sum_{j=1}^k \eye_j  S^{I}_{(n-2)} + 2(n-m)(n-2) \sum_{j=1}^k (\eye_j)_2 S^{I \downalphai{j}}_{(n-2)},\label{eq:partexp3} 
\end{equation}
as the number of tuples $(s,t,u,v, \hat{\sigma})$ where exactly one of $s,t \in
[m]$.

The final case is when both $s,t \in [m]$.  There are two subcases:  $u=v$ and
$u \neq v$.

If $u=v$, then for $\sigma$ to be $J$-separated $s,t$ must be in the same block of
$J$, say $J_j$, as $(s\; u) (t\;
u)$ will bring them into the same cycle of $\hat{\sigma}$.  If $u \in [m]$, then $u$ must be in the same block as $s$ and $t$ and $\sigma$ is
$J$-separated if and only if $\hat{\sigma}$ is $J \setminus \{s,t\}$-separated.
There are $\eye_j (\eye_j -1) (\eye_j -2)$ choices for $s,t$ and $u$ in this
case.  Thus, contributions when $u \in [m]$ total
\begin{equation}
	\sum_{j=1}^k (\eye_j)_3 S^{I \downarrow \downarrow
	(\eye_j)}_{(n-2)}\label{eq:partexp4},
\end{equation}
where ${I \downarrow \downarrow (\eye_j)}$ is the composition $I$ with $i_j$ replaced
by $i_j -2$.
If $u \notin [m]$, then $\sigma$ is $J$-separated if $\hat{\sigma}$ is
$J(s,t;u)$-separated, where $J(s,t;u)$ is $J$ with $s,t$ replaced by
$u$.  Since we have $\eye_j (\eye_j -1) (n-m)$ choices for $s,t$ and
$u$, the total contribution $u \notin [m]$ is
\begin{equation}
	\sum_{j=1}^k (\eye_j)_2 (n-m) S^{I \downarrow
	(\eye_j)}_{(n-2)}\label{eq:partexp5}.
\end{equation}

Finally, if $u \neq v$, 
first consider the case where $s$ and $t$ are in the same part of $J$.
Then $u$ and $v$ are either both in $[m]$, one is in $[m]$ and the
other is not, or neither is in $[m]$.
The contributions in each case are given by, respectively, a summand of
\begin{equation}
	\sum_{j=1}^k (\eye_j)_4 S^{I \downarrow \downarrow (\eye_j)}_{(n-2)}
	+ 2 \sum_{j=1}^k (\eye_j)_3 (n-m) S^{I \downalphai{j}}_{(n-2)} +
	\sum_{j=1}^k (\eye_j)_2 (n-m)_2 S^{I}_{(n-2)}.\label{eq:partexp6}
\end{equation}
If $s,t$ are not in the same part of $J$, then we have the following cases if
$\sigma$ is $J$-separated:
either $u,v \in [m]$, in which case $u$ is in the same block as $s$ and $v$ is in the same block as $t$ in $J$;
exactly one of $u$ or $v$ is in $[m]$, say $u$,
and then $u$ must be in the same block of $J$ as $s$;
or, $u,v \notin [m]$.
The total contribution from these cases is
\begin{equation}
	\sum_{j=1}^k \sum_{p=1 \atop p \neq j}^k (\eye_j)_2
	(\eye_p)_2 S^{I \downalphai{j} \downalphai{p}}_{(n-2)}
	+ 2 \sum_{j=1}^k \sum_{p=1 \atop p \neq j} \eye_j \eye_p (\eye_j -1) (n-m) S^{I
	\downalphai{j}}_{(n-2)}+
	\sum_{j=1}^k \sum_{p=1 \atop p \neq j}^k \eye_j \eye_p (n-m)_2 S^{I}_{(n-2)}  \label{eq:partexp7}.  
\end{equation}
Collecting like terms in \eqref{eq:partexp0} - \eqref{eq:partexp7}, we get
\begin{multline}
	\left((n-m)_2 (n-2)^2 + \sum_{j=1}^k \left( 2 (n-2) (n-m)_2  \eye_j  + (\eye_j)_2 (n-m)_2 + \eye_j
	\sum_{p \neq j} \eye_p (n-m)_2\right)\right) S^{I}_{(n-2)}\\
	+ \sum_{j=1}^k \left( 2 (\eye_j)_2 (n-m)(n-2) + (\eye_j)_2 (n-m) + 2 (\eye_j)_3 (n-m) + 2
	(\eye_j)_2 \sum_{p \neq j}(\eye_p) (n-m)\right) S^{I
	\downalphai{j}}_{(n-2)}\\
	+ \sum_{j=1}^k \left((\eye_j)_4 + (\eye_j)_3\right) S^{I \downarrow
	\downarrow (\eye_j)}_{(n-2)} + \sum_{j=1}^k (\eye_j)_2 (\eye_p)_2
	S^{I \downarrow (\eye_j) \downarrow
	(\eye_p)}_{(n-2)}.\label{eq:collected}
\end{multline}
The first sum in \eqref{eq:collected} simplifies as
\begin{align*}
	(n-m)_2 S^{I}_{(n-2)} & \left( (n-2)^2 + \sum_{j=1}^k \big( 2 (n-2)
	\eye_j +  (\eye_j)_2 + (\eye_j)(m-\eye_j) \big) \right)\\
	&= (n-m)_2 S^{I}_{(n-2)} \left( (n-2)^2 + \sum_{j=1}^k \left( 2 (n-2)
	\eye_j + \eye_j m  - \eye_j\right)\right)\\
	&= (n-m)_2 S^{I}_{(n-2)} \left( (n+m-2)^2 -m\right).
\end{align*}
The second sum above simplifies as
\begin{align*}
	(n-m) &\sum_{j=1}^k \left(2 (\eye_j)_2 (n-2) + 2(\eye_j)_3 + 2
	(\eye_j)_2 (m-\eye_j) + (\eye_j)_2\right) S^{I
	\downalphai{j}}_{(n-2)}\\
	=(n-m) &\left(2 n + 2 m  - 7\right)  \sum_{j=1}^k (\eye_j)_2 S^{I
	\downalphai{j}}_{(n-2)}\\
	=(n-m) &\left(2 n + 2 m  - 7\right) \sum_{j=1}^k (\eye_j)_2 (n-m -1)!\prod_{x
	\neq j} (\eye_x)! (\eye_1 - 1)! D^{m-1,k}_{n-2}\\
	=(n-m)!& \left(\prod (\eye_x)!\right)\left(2 n + 2 m  - 7\right)  D^{m-1,k}_{n-2} \left( \sum_{j=1}^k (\eye_j - 1) \right) \\
	= (n-m) ! &\left(\prod (\eye_x)!\right)  \left(2 n + 2 m  - 7\right) D^{m-1,k}_{n-2}  (m-k).
\end{align*}
For the third sum, we are mindful that $S^{I \downarrow \downarrow
(\eye_j)}_{(n-2)}$ is 0 if $\eye_j = 1$ and ${I \downarrow \downarrow
(\eye_j)}$ has one less part if $\eye_j = 2$.   These can both be taken care
of with the summation index and by noting that the coefficient is usually 0 in
such cases, and  likewise for the fourth term.  Thus, the third and fourth sum in \eqref{eq:collected} become
\begin{align*}
	\sum_{j=1}^k & \left( (\eye_j - 2)^2 (\eye_j)_2 \right) \prod_{x \neq j}
	(\eye_x)! (\eye_j - 2)! (n-m)! D^{m-2,k}_{n-2}\\
	&+ \sum_{j=1}^k \sum_{p \neq j}
	(\eye_j)_2 (\eye_p)_2 (n-m)! \prod_{x \neq p,j} (\eye_x)! (\eye_j
	- 1)! (\eye_p - 1)! D^{m-2,k}_{n-2}\\ 
	&= (n-m)! \left(\prod (\eye_x)!\right) D^{m-2,k}_{n-2} \left(\sum_{j=1 \atop \eye_j \neq 1} (\eye_j -
	2)^2 + \sum_{j = 1 \atop \eye_j \neq 1} (\eye_j - 1) (m - \eye_j -
	k + 1)\right)\\
	&= (n-m)! \left(\prod (\eye_x)!\right) D^{m-2,k}_{n-2} \left( (m-k-2)(m-k) + k -
	m_1(I) \right).
\end{align*}
Combining all the parts we have the following theorem.
\begin{theorem}\label{thm:bigthm}
	Suppose that $n$ is a positive integer and that $I = (i_1, \dots, i_k)$ is a composition of length $k$ of $m$,
	with $k,m \leq n$.  Then the separation probability $\sigma^I_{(n-1,1),
	(n-1,1)}$
	is given by
	\begin{multline*}
		\frac{(n-m)! \prod (\eye_j)!}{n^2 (n-2)!^2} \left( (k-m_1(I))
		D^{m-1,k}_{n-1} + m_1(I)D^{m-1,k-1}_{n-1} + D^{m,k}_{n-1}\right.\\
		+ \left.\left(
		(n+m-2)^2 - m\right) D^{m,k}_{n-2} + (m-k)(2n+2m-7)
		D^{m-1,k}_{n-2}\right.\\
		+\left.  \left( (m-k-2)(m-k) + k - m_1(I)
		\right)D^{m-2,k}_{n-2}\right).
	\end{multline*}
\end{theorem}
In \cite[Equation 1]{BernardiEtAlSeparation}, the authors state that
$\frac{\sigma^I_{\lambda, (n)}}{\prod (\eye_j)!}$ depends only on 
$n, \lambda,m$ and $k$.
For the special case when $\lambda = (n)$, this is apparent from the explicit
form of $\sigma^I_{(n),(n)}$ given in \eqref{eq:stanprob}.
In the same spirit, we have the following corollary.  Note the unexpected analogy with
the case $\rho = (1)$ in Theorem \ref{ThmDependence}. 
\begin{corollary}
	Suppose that $n \geq 1$ and $I = (\eye_1, \dots, \eye_k)$ is a
	composition of $m$.  Then $\frac{\sigma^I_{(n-1, 1), (n-1,1)}}{\prod (\eye_j)!}$
	depends only on $n,m,k$ and $m_1(I)$.
\end{corollary}

This result invites us to state the following conjecture.  We have computational evidence for
the following conjecture for small $a$ and $n$.

\begin{conjecture}\label{ConjGenSymmetry}
    Let $a$ and $n$ be two non-negative integers with $a < n$,
    $\lambda$ a partition of $n$
    and $I = (\eye_1, \dots, \eye_k)$ is a
    composition of $m$.
    Then $\frac{\sigma^I_{(n-a, 1^a), \lambda}}{\prod (\eye_j)!}$
    depends on $a,n,\lambda,m,k,m_1(I),\dots,m_{a-1}(I)$ but not
    the higher multiplicities of $I$.
\end{conjecture}

\bibliographystyle{alpha}
\bibliography{OnFactorizationNumbers}

\section*{Acknowledgements}

This research project was started during a visit of AR in Bordeaux.
This visit was funded {\em via} the ``invité junior'' program of LaBRI.
AR would like to thank people in LaBRI for their generous hospitality.

V.F. is partially supported by ANR grant PSYCO ANR-11-JS02-001.

Both authors would like to thank the anonymous referees for their helpful
suggestions.
\end{document}